\documentclass[11pt,a4paper]{amsart}
\usepackage[left=3cm,right=3cm,top=3cm,bottom=3cm]{geometry}
\usepackage{epsfig,amsthm,amsmath,amsfonts,amssymb,mathtools}
\usepackage{color,hyperref,lineno}
\usepackage{tikz,pgfplots}
\usepackage[export]{adjustbox}
\usetikzlibrary{arrows,decorations,backgrounds,math,shapes,patterns}
\usepgfplotslibrary{groupplots}
\pgfplotsset{compat=newest} 
\colorlet{shadecolor}{gray! 40}
\newtheorem{theorem}{Theorem}[section]
\newtheorem{lemma}[theorem]{Lemma}
\theoremstyle{definition}
\newtheorem{definition}[theorem]{Definition}

\theoremstyle{remark}
\newtheorem{remark}[theorem]{Remark}
\numberwithin{equation}{section}
\definecolor{mygreen}{RGB}{28,172,0} 
\definecolor{mylilas}{RGB}{170,55,241}

\usepackage{algorithm}
\usepackage[indLines=true,noEnd=true]{algpseudocodex}  
\newcommand{\bs}{\boldsymbol}

\newcommand{\spn}{\operatorname{span}}
\def\letters{a,b,c,d,e,f,g,h,i,j,k,l,m,n,o,p,q,r,s,t,u,v,w,x,y,z}
\def\Letters{A,B,C,D,E,F,G,H,I,J,K,L,M,N,O,P,Q,R,S,T,U,V,W,X,Y,Z}
\makeatletter
\@for \@l:=\Letters \do{%
  \expandafter\edef\csname\@l bb\endcsname{%
  \noexpand\ensuremath{\noexpand\mathbb{\@l}}}%
  \expandafter\edef\csname\@l bf\endcsname{{\noexpand\bf \@l}}%
  \expandafter\edef\csname\@l cal\endcsname{%
  \noexpand\ensuremath{\noexpand\mathcal{\@l}}}%
  \expandafter\edef\csname\@l eu\endcsname{%
  \noexpand\ensuremath{\noexpand\EuScript{\@l}}}%
  \expandafter\edef\csname\@l frak\endcsname{%
  \noexpand\ensuremath{\noexpand\mathfrak{\@l}}}%
  \expandafter\edef\csname\@l rm\endcsname{{\noexpand\rm \@l}}%
  \expandafter\edef\csname\@l scr\endcsname{%
  \noexpand\ensuremath{\noexpand\mathscr{\@l}}}%
}
\@for \@l:=\letters \do{%
  \expandafter\edef\csname\@l bf\endcsname{{\noexpand\bf \@l}}%
  \expandafter\edef\csname\@l frak\endcsname{%
  \noexpand\ensuremath{\noexpand\mathfrak{\@l}}}%
  \expandafter\edef\csname\@l scr\endcsname{%
  \noexpand\ensuremath{\noexpand\mathscr{\@l}}}%
}
\makeatother
\newcommand{\isdef}{\mathrel{\mathrel{\mathop:}=}}

\definecolor{green}{rgb}{0,0.5,0}

\begin{document}

\title[Kernel Interpolation on Generalized Sparse Grids]
{Kernel Interpolation on Generalized Sparse Grids}
\author{Michael Griebel}
\address{Michael Griebel,
Institut f\"ur Numerische Simulation,
Universit\"at Bonn, Friedrich-Hirzebruch-Allee 7, 53115 Bonn, Germany
and
Fraunhofer Institute for Algorithms and Scientific Computing (SCAI), 
Schloss Birlinghoven, 53754 Sankt Augustin, Germany
}
\email{griebel@ins.uni-bonn.de}
\author{Helmut Harbrecht}
\address{
Helmut Harbrecht,
Departement Mathematik und Informatik,
Universit\"at Basel, Spiegelgasse 1, 4051 Basel, Switzerland
}
\email{helmut.harbrecht@unibas.ch}
\author{Michael Multerer}
\address{
Michael Multerer,
Istituto Dalle Molle di studi sull’intelligenza artificiale, Universit\`a della Svizzera italiana,
Via la Santa 1, 6962 Lugano, Switzerland
}
\email{michael.multerer@usi.ch}


\begin{abstract}
We consider scattered data approximation on 
product regions of equal and different dimensionality. 
On each of these regions, we assume quasi-uniform 
but unstructured data sites and construct optimal sparse 
grids for scattered data interpolation on the product region. 
For this, we derive new improved error estimates for 
the respective kernel interpolation error by invoking
duality arguments. An efficient
algorithm to solve the underlying linear system 
of equations is proposed. The algorithm is based on the 
sparse grid combination technique, where a sparse
direct solver is used for the elementary anisotropic tensor 
product kernel interpolation problems. The application
of the sparse direct solver is facilitated by applying a
samplet matrix compression to each univariate kernel
matrix, resulting in an essentially sparse representation
of the latter. In this 
way, we obtain a method that is able to deal with large 
problems up to billions of interpolation points, especially
in case of reproducing kernels of nonlocal nature. 
Numerical results are presented to qualify and 
quantify the approach.
\end{abstract}
\maketitle

\section{Introduction}
Scattered data approximation using kernels is popular in many areas,
ranging from approximation theory to statistics. The approach facilitates the 
estimation of missing values in a dataset or to make predictions for new data 
sites based on the available data. Scattered data approximation is particularly
applied in imaging processing, surface reconstruction and machine learning, 
see for example \cite{Fasshauer,Schoelkopf,Theodoridis,Wendland} and the 
references therein. However, the naive computation of the
kernel approximate is known to suffer from the 
so-called \emph{curse of dimensionality} when the data dimension increases.

Various concepts exist to overcome the curse of dimensionality to a certain 
extent. A prominent approach is offered by \emph{sparse grids} or more general 
\emph{sparse tensor product spaces}, where the dimensions only mildly enter
in the cost estimates through a dimension-dependent power of a logarithmic 
factor, see \cite{BG,Smolyak,Zenger} for example. In this article, we aim 
at the construction and implementation of suitable sparse grids for the 
approximation of tensor product kernels. Interpreting kernel approximation 
in the context of Gaussian process learning, see \cite{RW}, the approach 
under consideration amounts to a multi-fidelity fusion model, see
e.g.~\cite{FLPW,PWG}, where the hierarchy of surrogate models is given 
by kernel approximates on a hierarchy of subspaces. A fundamental 
contribution to sparse grids for kernel approximation has recently 
been provided by \cite{Kempf1,Kempf2}. While the sparse grid construction therein 
relies on a multilevel approach invoking level dependent correlation lengths 
of the kernel function under consideration, we use here a kernel function of
fixed correlation length to construct the sparse grid interpolant. Especially, 
we discuss the optimality of the underlying sparse tensor product spaces and 
provide improved error estimates based on results in \cite{Schaback,Sloan}. 

The starting point for our construction is a 
tensor product Hilbert space
$$\bs{\mathcal{H}} = \bigotimes_{i=1}^m\mathcal{H}^{(i)},$$ 
formed by a finite collection of reproducing
kernel Hilbert spaces $\mathcal{H}^{(i)}$ with reproducing
kernels \(\kappa_i\), $i=1,\ldots,m$, defined on a
collection of bounded, Lipschitz-smooth regions 
$\Omega_i\subset\mathbb{R}^{d_i}$ of
relatively small and possibly different dimensions 
$d_i\in\mathbb{N}$. Associated to the tensor product 
reproducing kernel Hilbert space $\bs{\mathcal{H}}$,
we consider the product kernel $$\boldsymbol\kappa
({\bs x},{\bs y}) = \prod_{i=1}^m \kappa_i(x_i,y_i).$$ 
The kernel is the reproducing kernel of the space $\bs{\mathcal{H}}$ and
renders it itself a reproducing kernel Hilbert space defined on
the product region $\boldsymbol\Omega = \bigtimes_{i=1}^m\Omega_i$. 
Models of this type are applicable to multivariate interpolation
problems, where only scattered data are available within the unidirectional regions. 
Examples are environmental monitoring, multidimensional image and volume
reconstruction, such as magnetic resonance imaging, as well as simulation based
uncertainty quantification.

For the above setup, we construct an optimized sparse tensor product space 
to compute the kernel interpolant with respect to the underlying sparse grid. 
Employing results from \cite{GH1,GH2,Sloan}, we are able to derive improved 
error estimates for the sparse grid approximation error and related complexity 
bounds. To this end, we assume for each of the regions 
\(\Omega_i\) sets of quasi-uniform data sites. We propose
a simple algorithm to coarsen these 
sets in order to construct the
necessary multilevel hierarchy of approximation spaces for the sparse grid.
The implementation of the sparse grid and the computation of the
sparse grid interpolant is then based on the 
\emph{sparse grid combination technique} as introduced in \cite{GSZ,Smolyak}. 
This approach is known to successively 
compose the respective solution from the solutions to certain anisotropic
standard tensor product interpolation problems, see \cite{HPS,Kempf1,Kempf2}. 
We provide the details on the implementation of the sparse grid combination
technique as well as the storage and solution of the tensor product subproblems.
To solve the latter, we suggest the use of a direct solver 
that combines \emph{samplet matrix compression} with 
a \emph{sparse direct solver} as proposed in \cite{HM1,HM2}. This way,
the approach becomes computationally feasible, especially in case of 
nonlocal reproducing kernels. We present extensive numerical studies 
to qualify and quantify the approach.

The rest of the article is structured as follows: In 
Section~\ref{sec:prelim}, we introduce reproducing 
kernel Hilbert spaces and their basic theory. Then, in 
Section~\ref{sec:multivariate}, we define generalized sparse
grids and discuss their optimality concerning their complexity.
The numerical implementation and related algorithms are described
in Section~\ref{sct:implementation}. In Section~\ref{sct:numerix}, 
we perform numerical experiments which validate the present theory. 
Finally, in Section~\ref{sct:conclusio}, we draw some conclusions.

Throughout this article, to avoid the repeated use of unspecified 
generic constants, we write \(A \lesssim B\) if \(A\) is bounded 
by a uniform constant times \(B\), where the constant does not 
depend on any parameters which \(A\) and \(B\) might depend 
on. Similarly, we write \(A \gtrsim B\) if and only if \(B \lesssim A\).
Finally, if \(A \lesssim B\) and \(B \lesssim A\), we write \(A \sim B\).
Furthermore, the inequality \({\bs a}\leq{\bs b}\)
between two vectors has to be understood componentwise, i.e., \(a_i\leq b_i\)
for all \(i\). Likewise, \({\bs a}<{\bs b}\) means \(a_i < b_i\) for all \(i\).
\section{Preliminaries}\label{sec:prelim}
\subsection{Reproducing kernel Hilbert spaces}
Let $\Omega\subset\mathbb{R}^d$, $d\in\mathbb{N}$, 
be a Lipschitz-smooth region, which we assume to be 
bounded for the sake of simplicity. We start with the 
following definition:

\begin{definition}\label{def:RKHS}
A \emph{reproducing kernel} for a Hilbert space 
$\mathcal{H}$ of functions \(u\colon\Omega\to\mathbb{R}\)
with inner product $(\cdot,\cdot)_{\mathcal{H}}$ is a 
function $\kappa\colon\Omega\times\Omega\to\mathbb{R}$ such that
\begin{enumerate}
  \item $\kappa(\cdot,y)\in\mathcal{H}$ for all $y\in\Omega$,
  \item $u(y) = \big(u,\kappa(\cdot,y)\big)_\mathcal{H}$ 
  for all $u\in\mathcal{H}$ and all $y\in\Omega$.
\end{enumerate}
A Hilbert space $\mathcal{H}$ with reproducing kernel 
$\kappa\colon\Omega\times\Omega\to\mathbb{R}$ is called
\emph{reproducing kernel Hilbert space} (RKHS). 
\end{definition}

A continuous kernel $\kappa:\Omega\times\Omega\to\mathbb{R}$
is called \emph{positive semidefinite} on $\Omega\subset\mathbb{R}^d$
if 
\begin{equation}\label{eq:spd}
\sum_{i,j=1}^N \alpha_i\alpha_j
\kappa (x_i,x_j) \geq 0
\end{equation}
holds for all all mutually distinct
points $x_1,\ldots,x_N\in\Omega$ and 
all $\alpha_1,\dots,\alpha_N\in\mathbb{R}$, for any $N\in\mathbb{N}$. 
The kernel is even \emph{positive definite} if the inequality in 
\eqref{eq:spd} is strict whenever at least one \(\alpha_i\) 
is different from $0$.

Given a set $X = \{x_1,\ldots,x_N\}$ of $N$ mutually 
distinct data sites, we introduce the \emph{kernel translates}
$\phi_j \isdef \kappa(\cdot,x_j)$ for $j=1,\dots,N$. If the 
kernel \(\kappa\) is positive definite, these kernel translates 
span the $N$-dimensional subspace 
\[
\mathcal{H}_X \isdef  \spn\{\phi_1,\ldots,\phi_N\}\subset\mathcal{H}.
\]
The best approximation $f_X\in\mathcal{H}_X$ of a function
$f\in\mathcal{H}$ with respect to $\mathcal{H}$ amounts to its 
$\mathcal{H}$-orthogonal projection onto $\mathcal{H}_X$. The 
latter can be obtained as the solution of the variational 
formulation
\begin{equation}\label{eq:KerGalerkin}
  \text{find}\ f_X\in\mathcal{H}_X,\ \text{such that}\ 
  	(f_X,v)_{\mathcal{H}} = (f,v)_{\mathcal{H}}\ \ \text{for all $v\in\mathcal{H}_X$}.
\end{equation}
In view of the reproducing property, i.e., the second property from 
Definition~\ref{def:RKHS}, the ansatz $f_X = \sum_{i=1}^N 
\alpha_i\phi_i$ leads to the linear system of equations
\[
  {\bs K}\bs\alpha = \bs f
\]
with the kernel matrix
\begin{equation*}
  {\bs K} = \begin{bmatrix} 
  \kappa(x_1,x_1)&\cdots&\kappa(x_1,x_N)\\
  \vdots&\ddots&\vdots\\
  \kappa(x_N,x_1)&\cdots&\kappa(x_N,x_N)\end{bmatrix}
\end{equation*}
and the right-hand side $\bs f = 
[f(x_1),\dots,f(x_N)]^T$.

In particular, we observe that the resulting
system of equations coincides with the one for the
generalized Vandermonde matrix for the interpolation
at the data sites in $X$, i.e.
\[
  u(x_j) = \sum_{i=1}^N \alpha_i\phi_i(x_j) \overset{!}{=} f(x_j)
  \quad\text{for $j=1,\dots,N$}.
\]
This means that, within the RKHS framework, 
the best approximation $u\in\mathcal{H}_X$ of a function $f\in\mathcal{H}$ 
is given by the interpolant for the data sites $X$.
This is also referred to as \emph{kernel interpolation}.
Since kernel interpolation works on arbitrarily unstructured sets of 
data sites, 
it is often used to approximate scattered data. Such scattered data 
can be found in computer graphics, but also in machine learning of 
high-dimensional data sets, see e.g., \cite{Fasshauer,Wendland}.

\subsection{Error estimates}\label{subsct:error}
Having fixed the kernel of interest, we consider the problem of function 
approximation. We are interested in recovering an unknown function 
$ f \in \mathcal{H}$, given only a finite data set
\[
\{(x_1, f_1),\ldots,(x_N, f_N)\}\subset\Omega\times \mathbb{R}.
\]
We collect the data sites in the set $X\isdef \{ x_1,\dots, x_N \} \subset\Omega$. 
Associated to this set, we define two characteristic quantities,
namely the \emph{fill distance} 
\[
h_{X,\Omega} \isdef  \sup_{x \in \Omega} \min_{x_i \in X} \| x - x_i \|_2
\]
and the \emph{separation distance}
\[
q_X \isdef\min_{i \neq j} \| x_i - x_j \|_2.
\]
For the theoretical results presented later, we require that 
the set of data sites is \emph{quasi-uniform}, i.e., there 
is a constant $ c_{\operatorname{qu}} > 0 $ such that $q_X 
\leq h_{X,\Omega} \leq c_{\operatorname{qu}} q_X$. But note
that the subsequent error estimates do not require 
quasi-uniformity of $X$. Quasi-uniformity 
is merely required to bound the complexity, since then 
a comparison of volumes yields for the number $|X|=N$ of data
sites the relation $N\sim h_{X,\Omega}^{-d}$,
see, e.g., \cite[Proposition 14.1]{Wendland}.

If the norm in $\mathcal{H}$ is isomorphic to the 
norm in Sobolev space $H^s(\Omega)$ with $s>d/2$, 
i.e., if there holds $\|f\|_{\mathcal{H}}\sim
\|f\|_{H^s(\Omega)}$ for all $f\in\mathcal{H}$, 
then we have the following error estimate
\begin{equation}\label{eq:error1A}
  \|f-f_X\|_{L^2(\Omega)}\lesssim h_{X,\Omega}^s \|f\|_{H^s(\Omega)},
\end{equation}
compare \cite{Wendland}. If there even holds 
\begin{equation}\label{eq:HsCSU}
  (u,v)_{\mathcal{H}}
  	\lesssim \|u\|_{L^2(\Omega)} \|v\|_{H^{2s}(\Omega)}
\end{equation}
for all $u\in\mathcal{H}$ and $v\in H^{2s}(\Omega)$,
then using \cite[Theorem 1]{Sloan} we may double the 
rate of convergence with respect to $L^2(\Omega)$,
when the data provide additional smoothness in terms of 
$f\in H^{2s}(\Omega)$. For the reader's convenience, 
we recall the proof of the respective estimate here.

\begin{lemma}\label{lem:doubling}
Let $\Omega\subset\mathbb{R}^d$ be sufficiently smooth
and let \(f_X\) be the solution to \eqref{eq:KerGalerkin} 
with respect to \(\Hcal_X\subset\mathcal{H}\). Then, 
there holds
\begin{equation}\label{eq:error1B}
  \|f-f_X\|_{L^2(\Omega)}\lesssim h_{X,\Omega}^{2s} \|f\|_{H^{2s}(\Omega)}
\end{equation}
whenever $f\in {H^{2s}(\Omega)}$.
\end{lemma}

\begin{proof}
We apply \eqref{eq:error1A} to $g\isdef f-f_X$ and 
note that it belongs to $\mathcal{H}$ since $f\in H^{2s}(\Omega)
\subset\mathcal{H}$ and $f_X\in\mathcal{H}_X\subset\mathcal{H}$. 
In view of \eqref{eq:error1A}, we find
\[
\|g-g_X\|_{L^2(\Omega)}\lesssim 
h_{X,\Omega}^s\|g\|_{\mathcal{H}}.
\]
Since $g_X = f_X-f_X = 0$, this implies
\[
\|f-f_X\|_{L^2(\Omega)}\lesssim h_{X,\Omega}^s\|f-f_X\|_{\mathcal{H}}.
\]
We now conclude by the Galerkin orthogonality 
$f-f_X\perp_{\mathcal{H}}\mathcal{H}_X$ that
\begin{align*}
\|f-f_X\|_{L^2(\Omega)}^2&\lesssim h_{X,\Omega}^{2s}\|f-f_X\|_{\mathcal{H}}^2\\
&= h_{X,\Omega}^{2s}(f-f_X,f)_{\mathcal{H}}\\
&\lesssim h_{X,\Omega}^{2s}\|f-f_X\|_{L^2(\Omega)}\|f\|_{H^{2s}{\Omega)}}.
\end{align*}
The result follows now by dividing by the factor $\|f-f_X\|_{L^2(\Omega)}$.
\end{proof}

In what follows, we shall assume without loss of generality 
that \(\Hcal\) is equipped with an inner product such that \eqref{eq:HsCSU} 
holds.\footnote{An inner product that satisfies 
\eqref{eq:HsCSU} is constructed in Appendix~\ref{app:A}. Nonetheless,
the analysis presented in the following also applies with obvious 
modifications to the situation that \eqref{eq:HsCSU} does not hold. 
We refer the reader to Section~\ref{sct:conclusio} for the final 
result which is then obtained.} Then, from \eqref{eq:HsCSU} 
and \eqref{eq:error1B}, we can also derive an error estimate with 
respect to the energy space $\mathcal{H}$. By using again the 
orthogonality $f-f_X\perp_{\mathcal{H}}\mathcal{H}_X$, we conclude
\begin{align*}
  \|f-f_X\|_{\mathcal{H}}^2 &=
  	(f-f_X,f)_{\mathcal{H}}\\
  	&\lesssim \|f-f_X\|_{L^2(\Omega)}
  	\|f\|_{H^{2s}(\Omega)}\\
	&\lesssim h_{X,\Omega}^{2s}\|f\|_{H^{2s}(\Omega)}^2,
\end{align*}
which implies the desired error estimate with 
respect to the energy space, i.e.,
\begin{equation}\label{eq:error1C}
  \|f-f_X\|_{\mathcal{H}}\lesssim h_{X,\Omega}^s \|f\|_{H^{2s}(\Omega)}.
\end{equation}

In view of \eqref{eq:error1A}, \eqref{eq:error1B}, and
\eqref{eq:error1C}, we may employ standard interpolation arguments 
to summarize the above error estimates in accordance with
\begin{equation}\label{eq:error1}
  \|f-f_X\|_{H^t(\Omega)}\lesssim h_{X,\Omega}^{t'-t} \|f\|_{H^{t'}(\Omega)},
  \quad 0\le t\le s\le t' \le 2s.
\end{equation}

\subsection{Multilevel sequences}
We consider a sequence of quasi-uniform sets of
data sites
\begin{equation}\label{eq:multiscale}
  X_0\subset X_1\subset X_2\subset\cdots\subset\Omega
\end{equation}
such that $h_j \isdef  h_{X_j,\Omega}\sim 2^{-j}$ and, 
consequently, $|X_j|\sim 2^{jd}$. Associated
to the sequence of sets of data sites, we obtain the multilevel 
hierarchy of finite dimensional approximation spaces
\[
  \mathcal{H}_0\subset\mathcal{H}_1\subset\mathcal{H}_2
  	\subset\cdots\subset\mathcal{H},
\]
where we write $\mathcal{H}_j\isdef\mathcal{H}_{X_j}$ for
the sake of simplicity.

Let 
\begin{equation}\label{eq:orthProj}
P_j\colon\mathcal{H}\to\mathcal{H}_j
\end{equation}
denote the $\mathcal{H}$-orthogonal 
projection onto $\mathcal{H}_j$ and define the
\emph{detail projection}
\begin{equation}\label{eq:detProj}
Q_j \isdef  P_j-P_{j-1},
\end{equation} 
where we set $P_{-1}\isdef 0$, i.e., $Q_0 = P_0$. 
Fixing a maximum level \(J\in\Nbb\), the 
detail projections
$Q_j$ give rise to the $\mathcal{H}$-orthogonal decomposition
\[
  \mathcal{H}_J = \bigoplus_{j=0}^J \mathcal{W}_j,
  	\quad \text{where}\ \mathcal{W}_j \isdef  Q_j(\mathcal{H}).
\]
Especially, the error estimate \eqref{eq:error1} implies
\begin{equation}\label{eq:error2}
\begin{aligned}
  \|Q_j f\|_{H^t(\Omega)}
  &\le \|f-P_j f\|_{H^t(\Omega)} + \|f-P_{j-1} f\|_{H^t(\Omega)}\\
  &\lesssim h_j^{t'-t} \|f\|_{H^{t'}(\Omega)},
\end{aligned}
\end{equation}
for all $0\le t\le s\le t' \le 2s$ provided that $f\in H^{t'}(\Omega)$.

\section{Multivariate setting}
\label{sec:multivariate}
\subsection{Tensor product spaces}
\label{subsec:FullTensorProduct}
We consider $m\in\Nbb$ possibly distinct RKHS 
$\mathcal{H}^{(1)},\dots,\mathcal{H}^{(m)}$ with reproducing 
kernels $\kappa_1(x_1,y_1),\dots,\kappa_m(x_m,y_m)$ and 
associated regions $\Omega_1\subset\mathbb{R}^{d_1}, \dots,$ 
$\Omega_m\subset\mathbb{R}^{d_m}$, respectively. We are interested
in the efficient approximation of functions in the 
tensor product space
\[
  \bs{\mathcal{H}} \isdef  \bigotimes_{i=1}^m \mathcal{H}^{(i)}.
\]
Of course, this is again an RKHS with reproducing kernel
in product form
\[
  \bs\kappa({\bs x},{\bs y}) \isdef  \kappa_1(x_1,y_1)
  	\cdots\kappa_m(x_m,y_m),
\]
where ${\bs x} = (x_1,\dots,x_m),{\bs y} = (y_1,\dots,y_m)
\in{\bs\Omega}$ with $\bs\Omega\isdef\Omega_1\times\cdots\times\Omega_m$ 
denoting the $m$-fold 
product region.

For each \(i=1,\ldots,m\), we assume the existence of
a nested sequence of sets of data sites, i.e.,
\[
  X_0^{(i)}\subset X_1^{(i)}\subset X_2^{(i)}\subset\cdots\subset\Omega_i,
\]
such that $h_j^{(i)} \isdef  h_{X_j^{(i)},\Omega_i}\sim 2^{-j}$. This yields
associated multiscale hierarchies of finite dimensional approximation
spaces
\[
  \mathcal{H}_0^{(i)}\subset\mathcal{H}_1^{(i)}\subset\mathcal{H}_2^{(i)}
  	\subset\cdots\subset\mathcal{H}^{(i)},\quad i=1,\ldots,m,
\]
with $\mathcal{H}_j^{(i)} \isdef  \mathcal{H}_{X_j^{(i)}}^{(i)}$. 
Given a multi-index ${\bs j} = [j_1,\ldots,j_m]\in\mathbb{N}_0^m$, 
we can define the \emph{tensor product grid}
\[
  {\bs X}_{\bs j} \isdef  X_{j_1}^{(1)}\times\cdots\times X_{j_m}^{(m)}
  	\subset{\bs\Omega}
\]
with associated tensor product approximation space 
\[
  \bs{\mathcal{H}}_{\bs j} \isdef  \spn\{\bs\kappa(\cdot,{\bs x}):
  	{\bs x}\in {\bs X}_{\bs j}\} = \mathcal{H}_{j_1}^{(1)}\otimes
		\cdots\otimes\mathcal{H}_{j_m}^{(m)}\subset\bs{\mathcal{H}}.
\]

Given a function $f\in \bs{\mathcal{H}}$,
the kernel interpolant $f_{\bs j}\in \bs{\mathcal{H}}_{\bs j}$
with respect to the tensor product grid ${\bs X}_{\bs j}$ 
is retrieved by solving the linear system of equations
\begin{equation}\label{eq:tpsystem}
   {\bs K}_{\bs j}{\bs\alpha}_{\bs j} = {\bs f}_{\bs j}.
\end{equation}
Herein, the kernel matrix ${\bs K}_{\bs j}$ is defined as
the Kronecker product 
\[
{\bs K}_{\bs j}\isdef{\bs K}_{j_1}^{(1)}\otimes\cdots\otimes {\bs K}_{j_m}^{(m)}
\]
of the univariate kernel matrices 
\[
{\bs K}_{j_i}^{(i)} = [\kappa_i(x_k,
y_k)]_{x_k,y_k\in X_{j_i}^{(i)}},
\] 
while the right hand side is defined as
\(
{\bs f}_{\bs j} = [f({\bs x}_{\bs k})]_{{\bs x}_{\bs k}\in{\bs X}_{\bs j}}.
\)
Since the kernel interpolant is the best approximation
in each of the univariate subspaces, it is evident that
\({\bs u}_{\bs j}\) is the best approximation of
$u\in {\bs{\mathcal{H}}}$ in the subspace ${\bs{\mathcal{H}}}_{\bs j}$ 
with respect to the norm in ${\bs{\mathcal{H}}}$. Since the 
number of interpolation points 
\[
|{\bs X}_{\bs j}| = \prod_{i=1}^m \big|X_{j_i}^{(i)}\big| = \prod_{i=1}^m 2^{j_id_i}
\]
grows exponentially in $m$, the computation of \({\bs \alpha}_{\bs j}\) 
suffers from the curse of dimension.

\subsection{Sparse tensor product spaces}
A way to mitigate the curse of dimension is to
employ \emph{sparse tensor product approximation}. To this 
end, we introduce the $\bs{\mathcal{H}}_{\bs j}$-orthogonal
detail projections, cf.\ \eqref{eq:detProj}, 
\[  
   {\bs Q}_{\bs j}\colon
   \bs{\mathcal{H}}\to\bs{\mathcal{H}}_{\bs j},\quad
   {\bs Q}_{\bs j} \isdef  Q_{j_1}^{(1)}\otimes\cdots\otimes Q_{j_m}^{(m)},\quad{\bs j}\geq{\bs 0}.
\]
We assume that the univariate spaces $\mathcal{H}^{(i)}$ 
are equivalent to Sobolev spaces $H^{s_i}(\Omega_i)$ for all 
$i=1,\ldots,m$ and for the vector of Sobolev indices 
${\bs s} = [s_1,\ldots,s_m]^T$. Moreover, for 
${\bf t}=[t_1,\ldots,t_m]^T\ge{\bf 0}$, we 
introduce the tensor product Sobolev space
\[
{\bs H}^{\bs t}(\bs\Omega)\isdef
H^{t_1}(\Omega_1)\otimes\cdots\otimes H^{t_m}(\Omega_m).
\]
This tensor product 
Sobolev space is frequently also called Sobolev space of functions 
with dominating mixed derivatives. 

In view of \eqref{eq:error2}, we conclude by standard 
tensor product arguments the decay estimate
\begin{equation}\label{eq:estimate Q_j}
  \|{\bs Q}_{\bs j}f\|_{{\bs H}^{\bs t}(\bs\Omega)}
  	\lesssim {\bs h}_{\bs j}^{{\bs t}'-{\bf t}}
  		\|f\|_{{\bs H}^{{\bs t}'}(\bs\Omega)},
  			\quad {\bs 0}\le {\bs t}\le {\bs s}\le {\bs t}' \le 2{\bs s},
\end{equation}
where 
\[
{\bs h}_{\bs j}^{{\bs t}'-{\bf t}}
\isdef \big(h_{j_1}^{(1)}\big)^{t_1'-t_1}\cdots \big(h_{j_m}^{(m)}\big)^{t_m'-t_m},
\]
as usual for powers of vectors with matching dimensions.

Next, we define \emph{sparse tensor product spaces}. To this 
end, we introduce a weight vector ${\bf 0}<\bs{w} = [w_1,\ldots,w_m]^T$ such that 
$\|\bs{w}\|_{\infty} = 1$. The (weighted)
sparse tensor product space of level \(J\in\Nbb\) is then defined by
\[
  \widehat{\bs{\mathcal{H}}}_J^{\bs w}
  	= \bigoplus_{{\bs j}^T{\bs w}\le J} \bs{\mathcal{W}}_{\bs j},
  	\quad \text{where}\ \bs{\mathcal{W}}_{\bs j} 
		\isdef  {\bs Q}_{\bs j}(\bs{\mathcal{H}}).
\]
Corresponding to $\widehat{\bs{\mathcal{H}}}_J^{\bs w}$,
we define the \emph{sparse grid projection}
\[
  \widehat{\bs P}_J^{\bs w}\colon{\bs{\mathcal{H}}}
  	\to\widehat{\bs{\mathcal{H}}}_J^{\bs w},\quad \widehat{\bs P}_J^{\bs w}f
		= \sum_{{\bs j}^T{\bs w}\le J}\big(Q_{j_1}^{(1)}\otimes 
			\cdots\otimes Q_{j_m}^{(m)}\big)f,
\]
which yields the sparse grid kernel interpolant $\widehat{u}_J^{\bs w}
= \widehat{\bs P}_J^{\bs w}f\in\widehat{\bs{\mathcal{H}}}_J^{\bs w}$ 
of a given function $f\in\bs{\mathcal{H}}$.

\subsection{Error estimates}
In \cite{GH1,GH2}, the construction of generalized 
sparse tensor product spaces has been considered. 
Following the theory provided therein, we derive
the following results:

\begin{theorem}[Convergence]\label{thm:accuracy}
Let ${\bs 0}\le{\bs t}<{\bs s}<{\bs t}'\le 2{\bs s}$ 
and $f\in {\bs H}^{{\bs t}'}(\bs\Omega)$. Then, there 
holds the error estimate
\begin{equation}	\label{eq:convergence rate}
  \big\|f-\widehat{\bs P}_J^{\bs w}f\big\|_{{\bs H}^{\bs t}(\bs\Omega)}
    \lesssim 2^{-J\min\{\frac{t_1'-t_1}{w_1},\ldots,\frac{t_m'-t_m}{w_m}\}} 
    J^{P-1}\|f\|_{{\bs H}^{{\bs t}'}(\bs\Omega)}.
\end{equation}
Here, $P\in\Nbb$ counts how often the minimum is attained 
in the exponent.
\end{theorem}

\begin{proof}
We have by the triangle inequality and by 
\eqref{eq:estimate Q_j} that
\[
   \big\|f-\widehat{\bs P}_J^{\bs w}f\big\|_{{\bs H}^{\bs t}(\bs\Omega)}
   \le\sum_{{\bs j}^T{\bs w}>J} \|{\bs Q}_{\bs j}f\|_{{\bs H}^{\bs t}(\bs\Omega)}
   \lesssim \sum_{{\bs j}^T{\bs w}>J} {\bs h}_{\bs j}^{{\bs t}-{\bf t}'}
  		\|f\|_{{\bs H}^{\bs t}(\bs\Omega)}.
\]
Due to $h_{j_i}^{(i)} \sim 2^{-j}$ for all $i=1,2,\ldots,m$ and hence 
${\bs h}_{\bs j}\sim 2^{-|{\bs j}|}$, we can now follow line-by-line 
the proof of \cite[Theorem~4.3]{GH2} and obtain the desired estimate.
\end{proof}

\begin{remark}\label{rem:logs}
Estimate \eqref{eq:convergence rate} remains valid without the
logarithmic factor $J^{P-1}$ in the case ${\bs t} = {\bs t}' = {\bs s}$ due 
to the Galerkin orthogonality in accordance with
\[
   \big\|f-\widehat{\bs P}_J^{\bs w}f\big\|_{\bs{\mathcal{H}}}^2
   =\sum_{{\bs j}^T{\bs w}>J} \|{\bs Q}_{\bs j}f\|_{\bs{\mathcal{H}}}^2
   \le\|f\|_{\bs{\mathcal{H}}}^2.
\]
As a consequence, if ${\bs t} = {\bs s}$ and ${\bs s} < {\bs t}'$, 
the logarithmic factor in \eqref{eq:convergence rate} is only 
$J^{(P-1)/2}$. Likewise, by applying the Aubin-Nitsche lemma, 
one concludes only the factor $J^{(P-1)/2}$ if ${\bs t} < {\bs s}$ and 
${\bs s} = {\bs t}'$, which improves the result of \cite{Kempf1,Kempf2}.
\end{remark}

We shall next count the degrees of freedom, i.e.,
the dimension, of the sparse 
tensor product space $\widehat{\bs{\mathcal{H}}}_J^{\bs w}$.

\begin{theorem}[Complexity]	\label{thm:complexity}
For any ${\bs w}>{\bf 0}$, the dimension of the sparse 
tensor product space $\widehat{\bs{\mathcal{H}}}_J^{\bs w}$
is proportional to $2^{J\max\{d_1/w_1,\ldots,d_m/w_m\}} J^{R-1}$, 
where $R\in\Nbb$ counts how often the maximum is attained.
\end{theorem}

\begin{proof}
In view of $\dim\mathcal{H}_{j_i}^{(i)} = 2^{j_i d_i}$ for all
$i=1,2,\ldots,m$, the assertion follows by nearly verbatim 
rewriting the proof of \cite[Theorem~4.1]{GH2}.
\end{proof}

As shown in \cite{GH2}, the combination of Theorems
\ref{thm:accuracy} and \ref{thm:complexity} yields 
the following estimate on the cost-complexity of the 
approximation in the sparse tensor product space
$\widehat{\bs{\mathcal{H}}}_J^{\bs w}$:

\begin{theorem}[Cost-complexity rate]\label{thm:cost complexity}
Let ${\bs 0}\le{\bs t}<{\bs s}<{\bs t}'\le 2{\bs s}$ and 
$f\in {\bs H}^{{\bs t}'}(\bs\Omega)$. Furthermore, denote 
by $N \isdef\dim\widehat{\bs{\mathcal{H}}}_J^{\bs w}$ the number 
of degrees of freedom in the sparse tensor product space 
$\widehat{\bs{\mathcal{H}}}_J^{\bs w}$ and set
\[
  \beta \isdef \frac{\min\{(t_1'-t_1)/w_1,\ldots,(t_m'-t_m)/w_m\}}
  		{\max\{d_1/w_1,\ldots,d_m/w_m\}}.
\]
Assume that the minimum in the enumerator is attained 
$P\in\Nbb$ times and the maximum in the denominator is 
attained $R\in\Nbb$ times. Then, the sparse grid kernel 
interpolant in $\widehat{\bs{\mathcal{H}}}_J^{\bs w}$
satisfies the error estimate
\begin{equation}	\label{eq:cost rate}
\big\|f-\widehat{\bs P}_J^{\bs w}f\big\|_{{\bs H}^{\bs t}(\bs\Omega)}
  	\lesssim N^{-\beta}(\log N)^{(P-1)+\beta(R-1)}
		\|f\|_{{\bs H}^{{\bs t}'}(\bs\Omega)}
\end{equation}
in terms of the degrees of freedom $N$.
\end{theorem}

It has been shown in \cite[Lemma 5.1]{GH2} that there holds
\[
  \beta\le\beta^\star\isdef \min\bigg\{\frac{t_1'-t_1}{d_1},\ldots,\frac{t_m'-t_m}{d_m}\bigg\}
\]
for all $\bs w>{\bf 0}$. Moreover, if the above minimum is
attained for the index $1\le\ell\le m$, then we achieve the
maximum rate $\beta = \beta^\star$ in \eqref{eq:cost rate}
for all $\bs w>{\bf 0}$ such that
\begin{equation}	\label{eq:inequality}
  \frac{t_\ell'-t_\ell}{t_i'-t_i}\le\frac{w_\ell}{w_i}\le\frac{d_\ell}{d_i}
		\quad\text{for all $i=1,2,\ldots,m$}.
\end{equation}
Natural choices of the parameter ${\bs w}>{\bf 0}$ are:
\renewcommand{\theenumi}{\emph{\roman{enumi}.}}
\begin{enumerate}
\item
To equilibrate the accuracy in the extremal univariate spaces
$\Hcal_{J/w_i}^{(i)}$, $i=1,2,\ldots,m$, we obtain the condition
\[
  2^{-J(t_1'-t_1)/w_1} = 2^{-J(t_2'-t_2)/w_2} 
  	= \cdots = 2^{-J(t_m'-t_m)/w_m}. 
\]
This means that we have to choose $\widetilde{w}_i
\isdef t_i'-t_i$ for all $i=1,2,\ldots,m$ and then rescale
${\bs w} \isdef \widetilde{\bs w}/\|\widetilde{\bs w}\|_{\infty}$.
This choice corresponds to the lower bound in \eqref{eq:inequality}.

\item
To equilibrate the number of degrees of freedom in the extremal 
univariate spaces $\Hcal_{J/w_i}^{(i)}$, $i=1,2,\ldots,m$, we obtain 
the condition
\[
  2^{Jd_1/w_1} = 2^{Jd_2/w_2} = \cdots = 2^{Jd_m/w_m}.
\]
This condition is satisfied if $\widetilde{w}_i\isdef d_i$ 
for all $i=1,2,\ldots,m$ and then setting ${\bs w} 
\isdef \widetilde{\bs w}/\|\widetilde{\bs w}\|_{\infty}$.
This choice yields the upper bound in \eqref{eq:inequality}.

\item
Following the idea of an {\em equilibrated cost-benefit rate} 
(see \cite{BG}), we get the condition
\[
  2^{j_1(d_1+t_1'-t_1)}\cdot 2^{j_2(d_2+t_2'-t_2)}\cdots 2^{j_m(d_m+t_m'-t_m)} 
    = 2^{J\cdot const.}
\]
for all ${\bs j}^T{\bs w} = J$. For $const. = 1$, 
we find $\widetilde{w}_i=d_i+t_i'-t_i$ for all 
$i=1,2,\ldots,m$. By setting again ${\bs w} \isdef 
\widetilde{\bs w}/\|\widetilde{\bs w}\|_{\infty}$
we derive a weight ${\bs w}$ which is between the
lower and upper bound in \eqref{eq:inequality} provided
that these differ from each other.
\end{enumerate}
We like to emphasize that the equilibration of the 
degrees of freedom is the only choice which gives always
the highest rate $\beta^\star$ (except for polylogarithmic 
factors), independent of the kernel under consideration 
or the particular smoothness of the function to be 
approximated. We refer the reader to \cite{GH2} for
a more detailed discussion.

\subsection{Comparison of sampling rates}
\label{sct:comparison}
We now want to put our result into perspective. In the 
regular sparse grid case on the unit $m$-cube $\bs\Omega 
= [0,1]^m$ and a product kernel that belongs to an RKHS 
being equivalent to $H^s([0,1])$, i.e.
\[
d_1 = d_2 = \cdots = d_m = 1, \quad s_1 = s_2 = \cdots = s_m = s,
\]
the upper and lower bound coincide and the only 
optimal weight is 
\[
  w_1 = w_2 = \cdots = w_m = 1. 
\]
It is well known that the standard Smolyak construction
without exploiting orthogonality gives
\[
\big\|f-f_N^{\text{Smolyak}}\big\|_{L^2(\bs\Omega)}
  	\lesssim N^{-s}(\log N)^{(s+1)(m-1)}
		\|f\|_{{\bs H}^{{\bs s}}(\bs\Omega)},
\]
see \cite{Smolyak}. But we can now exploit the orthogonality 
with respect to the RKHS in the error estimate as outlined 
in Remark \ref{rem:logs}. Hence, \eqref{eq:convergence rate} 
has only the logarithmic power $(P-1)/2$ instead of $P-1$ 
for ${\bf t} = {\bf 0}$ and ${\bf t}' = {\bf s}$. Thus, since 
$P = R = m$ and $\beta=s$, the respective cost-complexity 
rate for a function $f\in {\bs H}^{\bs s}(\bs\Omega)$ is
\begin{equation}\label{eq:our_rate}
\big\|f-\widehat{\bs P}_J^{\bs w}f\big\|_{L^2(\bs\Omega)}
  	\lesssim N^{-s}(\log N)^{(s+1/2)(m-1)}
		\|f\|_{{\bs H}^{{\bs s}}(\bs\Omega)}.
\end{equation}
Note at this point that it is known from \cite{Krieg} 
that there exists a set of $N$ points such that the 
best possible sampling rate would be given by
\[
 \big\|f-f_N^{\text{best}}\big\|_{L^2(\bs\Omega)}
  	\lesssim N^{-s}(\log N)^{s(m-1)}
		\|f\|_{{\bs H}^{{\bs s}}(\bs\Omega)}.
\]
This approach is however not constructive and such optimal 
point sets are not yet computable. The currently 
best point sets which are constructable provide the rate
\begin{equation}\label{eq:ullrich}
 \big\|f-f_N^{\text{constructive}}\big\|_{L^2(\bs\Omega)}
  	\lesssim N^{-s}(\log N)^{s(m-1)+1/2}
		\|f\|_{{\bs H}^{{\bs s}}(\bs\Omega)}.
\end{equation}
compare \cite{Ullrich}. This rate can be seen from (1.8) in \cite{Ullrich} 
and the linear widths for Sobolev spaces of bounded mixed derivatives 
${\bf H}^{\bf s}(\bs\Omega)$ in \cite[p.~46]{Dung1}. We should emphasize
that such point sets have to be computed in an offline phase that has
runtime $\mathcal{O}(N^3)$.

The cost-complexity rate \eqref{eq:ullrich} is the same as for 
our sparse grid point sets in \eqref{eq:our_rate} for the case 
$m=2$ and is indeed better for $m>2$ by an additive factor $(m-2)/2$ 
in the exponent of the logarithmic term. However, the huge 
practical advantage of sparse grid points over more general point 
sets is that the point distributions are structured which can be 
exploited to speed-up computations considerably. Moreover, the 
parallelization of the implementation based on the sparse grid
combination technique is straightforward.

\subsection{Sparse grid combination technique}
Due to the Galerkin orthogonality, it is easy to see that 
the detail projections satisfy
\[
  (\bs Q_{\bs j}u,\bs Q_{{\bs j}'}v)_{\bs{\mathcal{H}}}
    = 0 \quad\text{for ${\bs j}\not= {\bs j}'$ and any $u,v\in\bs{\mathcal{H}}$}.
\]
Therefore, the subspaces $\bs{\mathcal{W}}_{\bs j}$ and 
$\bs{\mathcal{W}}_{{\bs j}'}$ are $\bs{\mathcal{H}}$-perpendicular. 
Thus, since the kernel under consideration is of product type, 
the theory of \cite{HPS} tells us that we can compute the kernel 
interpolant in the sparse tensor product space 
$\widehat{\bs{\mathcal{H}}}_J^{\bs w}$ by 
means of the combination technique. 

With this in mind, we define the tensorized version of the
orthogonal projections \eqref{eq:orthProj} given by
\[
 {\bs P}_{\bs j}\colon\bs{\mathcal{H}}\to\bs{\mathcal{H}}_{\bs j},\quad
 {\bs P}_{\bs j} \isdef  P_{j_1}^{(1)}\otimes\cdots\otimes P_{j_m}^{(m)},
\]
and note that there holds the identity
\[
 {\bs P}_{\bs j} = \sum_{{\bs\ell}\le{\bs j}} {\bs Q}_{\bs\ell}.
\]
Moreover, in accordance with \cite{D1,D2,GSZ,HHPS,Smolyak}, 
we introduce the (weighted) \emph{combination technique index set}
\begin{equation}\label{eq:anisoset}
\mathcal{J}_J^{\bs w}\isdef \big\{\bs j\in\mathbb{N}_0^m:
J-|{\bs w}|<\bs {\bs j}^T{\bs w}\le J\big\}.
\end{equation}
With these definitions set at hand, one has the identity
\begin{equation}\label{eq:combiformel}
\widehat{\bs P}_J^{\bs w} = \sum_{{\bs j}\in\mathcal{J}_J^{\bs w}}
c_{\bs j}^{\bs w} {\bs P}_{\bs j},
\quad\text{where }
c_{\bs j}^{\bs w}\isdef \sum_{\genfrac{}{}{0pt}{}{{\bs j}'\in\{0,1\}^m}
{({\bs j}+{\bs j}')^T{\bs w}\le J}}(-1)^{|{\bs j}'|}.
\end{equation}
Hence, the sought sparse grid
kernel interpolant $\widehat{u}_J^{\bs w}
= \widehat{\bs P}_J^{\bs w}f\in\widehat{\bs{\mathcal{H}}}_J^{\bs w}$ 
is composed by the tensor product
kernel interpolants $u_{\bs j} = {\bs P}_{\bs j} f$ 
from different full tensor product spaces $\bs{\mathcal{H}}_{\bs j}$. 
Each of these tensor product kernel interpolants $u_{\bs j}$ 
can now be computed in accordance 
with Subsection~\ref{subsec:FullTensorProduct}.

\section{Implementation}\label{sct:implementation}
\subsection{Construction of nested point sets}
\label{sct:coarsening}
In this section, we comment on our implementation of the 
sparse grid kernel interpolation. We first describe how
we generate the multilevel sequence \eqref{eq:multiscale}
from a given set of quasi-uniform data sites $X\subset\Omega$.
Then, since each particular term in the sparse grid 
combination technique amounts to the solution of a dense 
linear system of equations which is of tensor product 
structure, we apply tensorization methods. Moreover, 
we use a fast method for nonlocal operators for each
subproblem that is associated to direction $i$, where 
$i=1,\ldots,m$. As we will demonstrate by numerical experiments, 
we altogether obtain a very efficient method to compute 
the sparse grid kernel interpolant.

\begin{figure}[htb]
\begin{center}
\includegraphics[scale = 0.072,clip,trim=380 230 380 230]{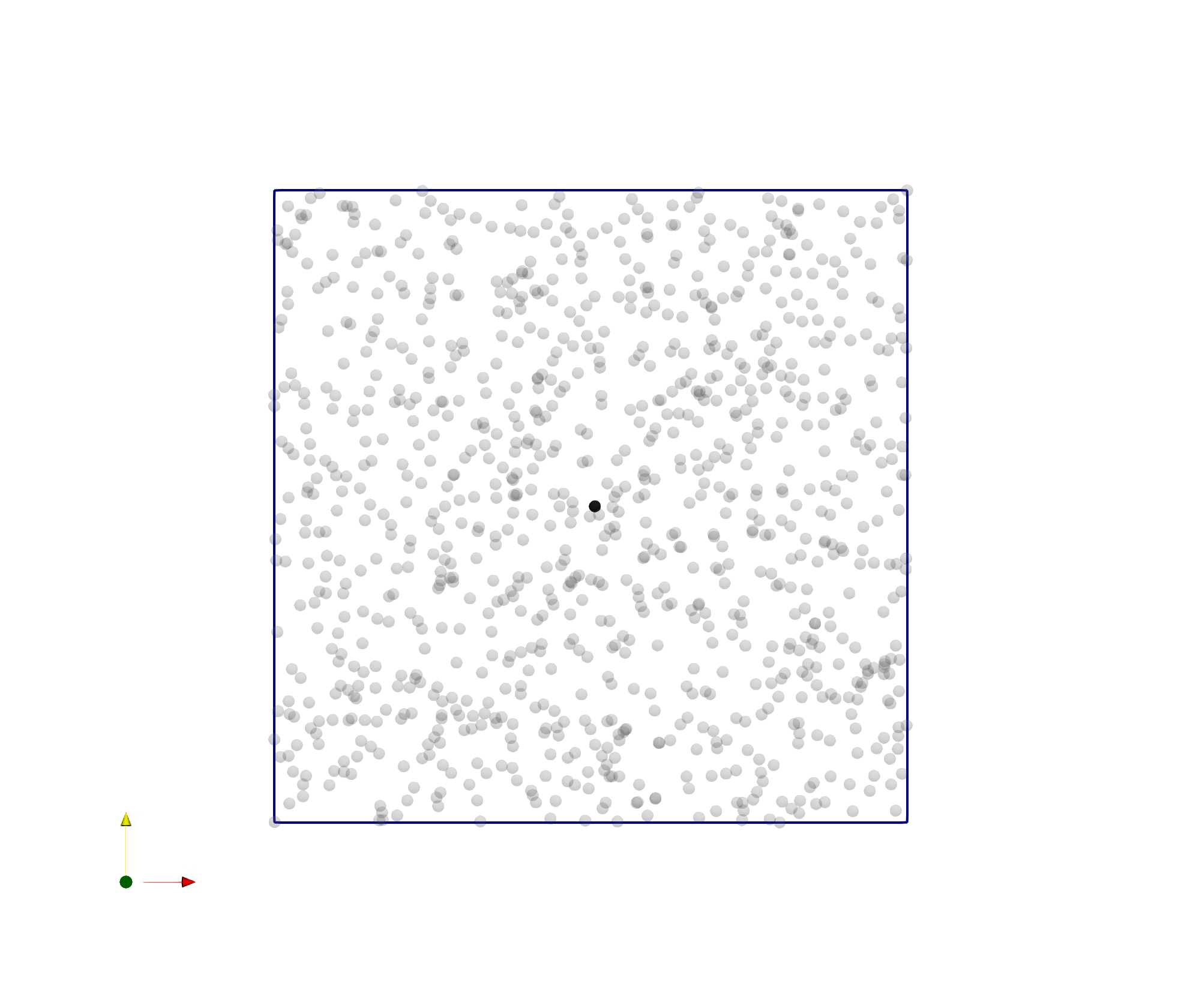}\quad
\includegraphics[scale = 0.072,clip,trim=380 230 380 230]{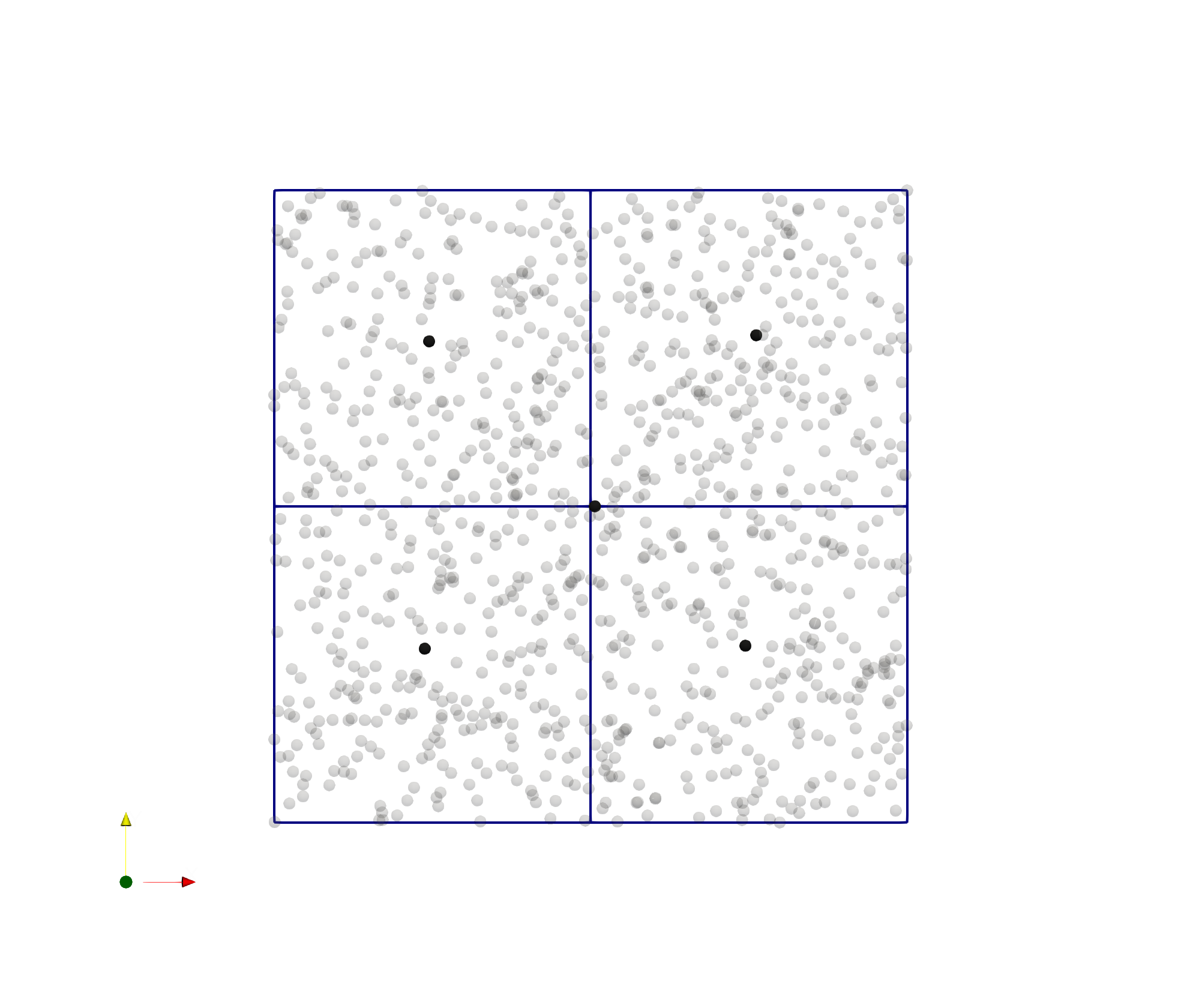}\quad
\includegraphics[scale = 0.072,clip,trim=380 230 380 230]{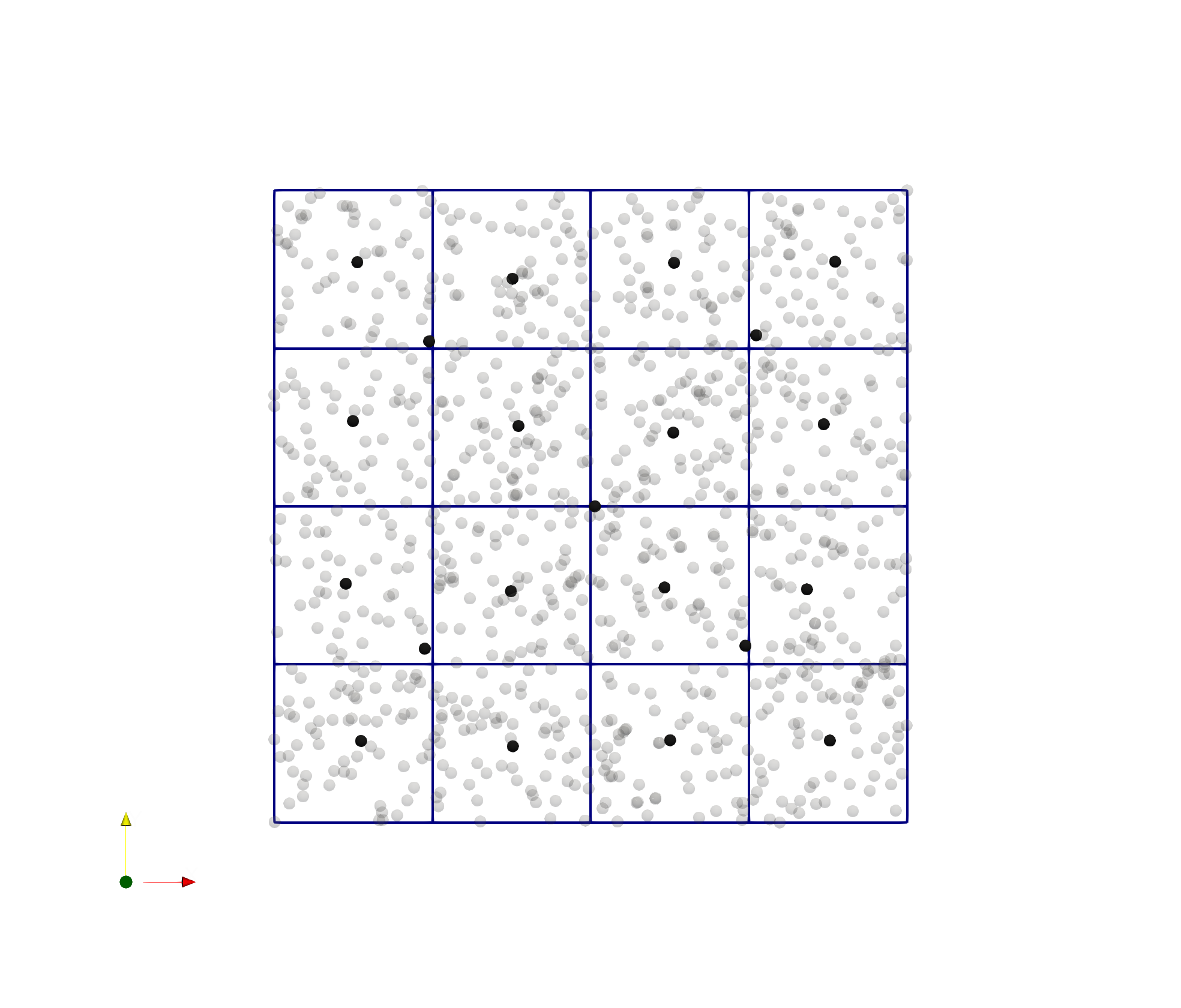}\quad
\includegraphics[scale = 0.072,clip,trim=380 230 380 230]{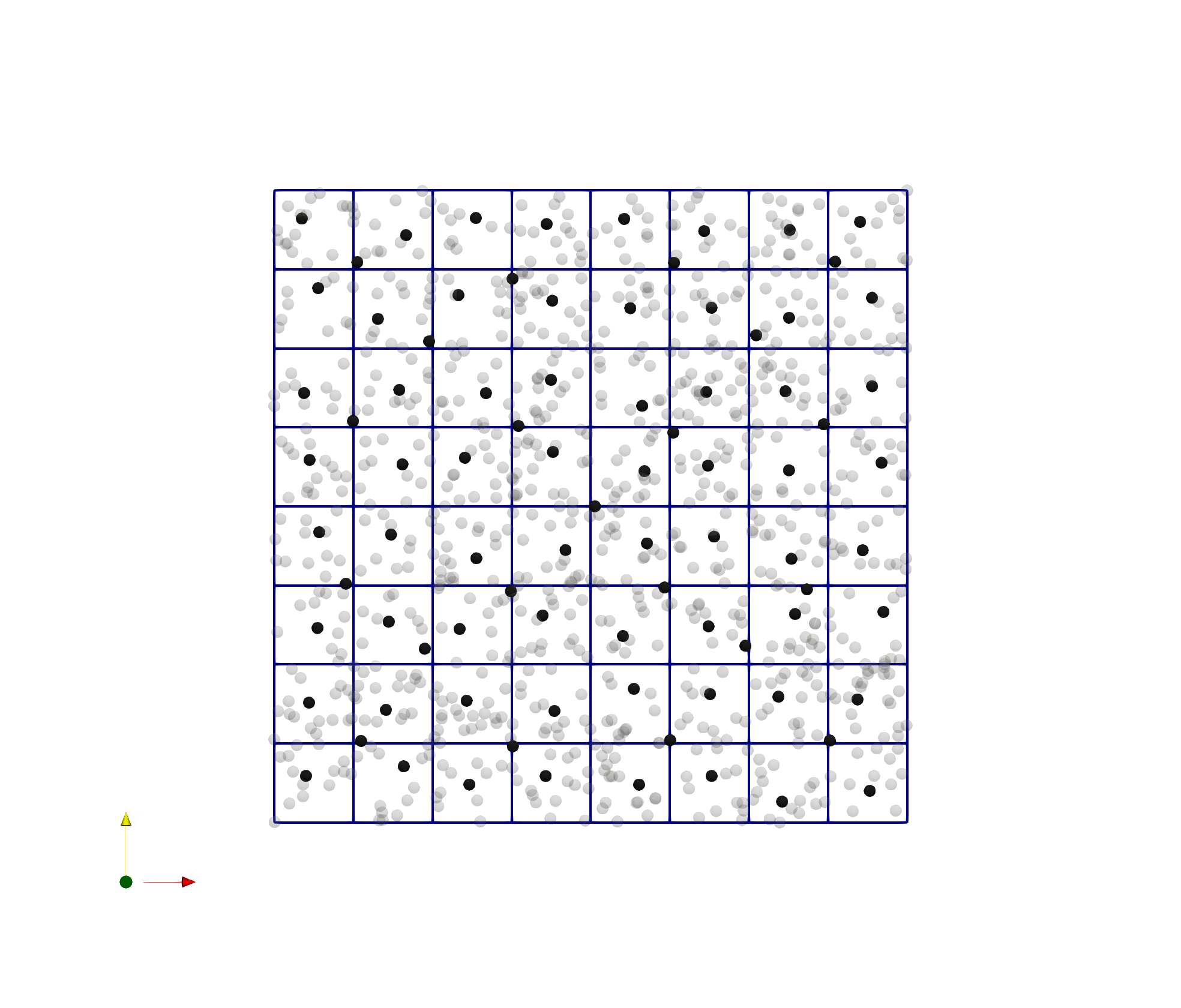}
\end{center}
\caption{\label{fig:subsample}Visualization of the subsampling 
procedure starting from a set of 1000 uniformly chosen random 
points on $[0,1]^2$.}
\end{figure}

Different sophisticated algorithms for 
the construction of nested subsets from a given set of data 
sites have been proposed in the literature, see, e.g., 
\cite{DeF89,FI96,SHD11}. Nonetheless, for our our purposes, 
the simple algorithm described below is sufficient.
To construct a multilevel sequence \eqref{eq:multiscale}
for a given set of quasi-uniform data sites $X$
and a given maximum level \(J\in\Nbb\),
we assume without loss of generality that
\(\Omega\subset[0,1]^d\). Otherwise, \(\Omega\) can be mapped
into \([0,1]^d\) by an affine transform and the subsequent
procedure has to be adapted accordingly.

We apply the following top-down 
algorithm: For each level $j=0,1,\ldots,J$, we subdivide 
$[0,1]^d$ equidistantly into $2^{jd}$ cuboids of edge length 
$2^{-j}$. To determine the point set \(X_j\), we start from
\(X_{j-1}\) and add points that are not already contained in
this set. To that end, from all points that are in a given 
cuboid, the point which is closest to the midpoint of the 
cuboid is chosen (in case of nonuniqueness, one randomly 
chooses one of the closest points). Thus, if each cuboid's 
intersection with the region \(\Omega\) contains at least one 
point, a fill distance \(h_{X_j,\Omega}\sim 2^{-j}\) is guaranteed. 
We remark that this is already achieved by taking any 
point within each cuboid. But choosing the point closest 
to the midpoint has the advantage of improving the separation 
distance. A visualization of the subsampling procedure for 
\(j=0,1,2,3\), starting from a set of 1000 uniformly 
chosen random points on $[0,1]^2$, is given in 
Figure~\ref{fig:subsample}.

\begin{algorithm}[htb]
\caption{\label{alg:UnifSubs}Uniform Subsample}
\begin{algorithmic}[1]
\Function{uniformSubsample}{$\Ical, X, j$}
    \State $\Ical^c\gets [1,\ldots,|X|]\setminus\Ical$
    \State $\Ical_{\text{new}}\gets\emptyset$
    \ForAll{$p\in\Ical^c$}
        \State${\bs m} \gets 2^{-j}(\lfloor2^{j}{\bs x}_p\rfloor+{\bf 0.5})$
        \If{$p_{\bs c}\in \Ical_{\text{new}}$}
        \If{$\|{\bs x}_{p_{\bs c}}-{\bs m}\|>\|{\bs x}_{p}-{\bs m}\|$}
                \State $p_{\bs c}\gets p$
                \EndIf
        \Else
        \State $p_{\bs c}\gets p$
        \State $\Ical_{\text{new}}\gets\Ical_{\text{new}}\cup\{p_{\bs c}\}$
        \EndIf
          \EndFor
    \State \Return $\Ical\cup\Ical_{\text{new}}$
\EndFunction
\end{algorithmic}
\end{algorithm}

An implementation can be found in Algorithm~\ref{alg:UnifSubs}.
It updates a given index set \(\Ical\) by selecting associated 
points as described above which are not already in \(\Ical\). 
Starting from \(\Ical=\emptyset\) and iterating then for 
\(j=0,\ldots,J\) results in the desired multilevel hierarchy. 
The cost of the algorithm for each level \(j\) is linear in the 
cardinality of $X$. 

\subsection{Computing the sparse grid kernel interpolant}
For the computation of the sparse grid kernel interpolant, 
we rely on the combination technique \eqref{eq:anisoset}. For 
each multi-index \(\bs j\in\Jcal_{J}^{\bs w}\), we have to 
solve the tensor product linear system 
\begin{equation}\label{eq:LGS}
{\bs K}_{\bs j} {\bs \alpha}_{\bs j} = {\bs f}_{\bs j}\quad\text{with}
\quad {\bs K}_{\bs j} = {\bs K}_{j_1}^{(1)}\otimes\cdots\otimes {\bs K}_{j_m}^{(m)},
\end{equation}
compare~\eqref{eq:tpsystem}. 
To exploit the tensor product structure of 
the linear system \eqref{eq:LGS}, we need a suitable 
representation of the quantities \({\bs \alpha}_{\bs j}\)
and \({\bs f}_{\bs j}\) in the computer. Moreover, we need 
to be able to unfold the tensor linear system \eqref{eq:LGS} 
to conventional linear systems ${\bs K}_{j_i}^{(i)} 
{\bs \alpha}_{j_i}^{(i)} = {\bs f}_{j_i}^{(i)}$
which belong to the directions $i=1,\ldots,m$, and 
are to be solved successively. Finally, we need 
a backtransform of the resulting solutions 
${\bs \alpha}_{j,i}^{(i)}$ to their associated tensor 
representation ${\bs \alpha}_{\bs j}^{(i)}$. Such tensor 
methods have become important tools in the recent years, 
see \cite{Hackbusch} for example, and can be applied 
in our context.

\begin{algorithm}[htb]
\caption{\label{algo:ClassTensor}Class Tensor}
\begin{algorithmic}[1]
\Function{toScalarIndex}{${\bs k},{\bs n}\in\Nbb_0^m$}
\State ${\bs b}\gets $\,\Call{strides}{${\bs n}$}
    \State $p \gets 0$
    \For{$i = 1,\ldots, m$}
        \State $p\gets p+ k_ib_i$
    \EndFor
    \State \Return $p$
\EndFunction
\Statex \hrulefill
\Function{toMultiIndex}{$p\in\Nbb_0,{\bs n}\in\Nbb_0^m$}
\State ${\bs b}\gets $\,\Call{strides}{${\bs n}$}
    \State ${\bs k}\gets{\bs 0}$
    \For{$i = 1,\ldots, m$}
        \State $(k_i,p) \gets (p / b_i,p\! \mod b_i)$
    \EndFor
    \State \Return ${\bs k}$
\EndFunction
\Statex \hrulefill
\Function{matricize}{$k\in\Nbb_0,o,p\in\Nbb,{\bs n}\in\Nbb_0^m$}
\State ${\bs b}\gets $\,\Call{strides}{${\bs n}$}
    \State $z \gets 0$
    \State $r \gets p$
    \For{$i = 1,\ldots, k-1$}
        \State $s \gets b_i / n_k$
        \State $(c,r)\gets (\lfloor r/s\rfloor, r\! \mod s)$
        \State $z \gets z + cb_i$
    \EndFor
    \State $r \gets r + ob_k$
    \For{$i = k + 1,\ldots,m$}
    \State $(c,r)\gets(\lfloor r/b_i\rfloor, r\! \mod b_i)$        
        \State $z \gets z + cb_i $
    \EndFor
    \State \Return $z$
\EndFunction
\Statex \hrulefill
\Function{strides}{${\bs n}$}
    \State ${\bs b}={\bs 0}$
    \For{$i=1,\ldots, m$}
        \State $b_i \gets \prod_{o=i+1}^{m} n_o$
    \EndFor
    \State \Return ${\bs b}$
\EndFunction
\end{algorithmic}
\end{algorithm}

The class \textsc{Tensor} in Algorithm~\ref{algo:ClassTensor} 
provides an implementation of an elementwise serialization of a given 
tensor in main memory by means of the method \textsc{toScalarIndex}. 
Given a multi-index \({\bs k}\in\{0,n_1\}\times\cdots\times
\{0,n_m\}\), the function assigns a unique linear index 
\(p=\text{\textsc{toScalarIndex}}
({\bs k},{\bs n})\in\{0,\ldots,\prod_{i=1}^mn_i\}.\)
This is achieved by a mixed radix representation, with the 
basis generated by the method \textsc{strides}. 
The corresponding inverse mapping from a scalar index
to a multi-index is given by \textsc{toMultiIndex}, 
which amounts to the Euclidean division algorithm.

Now, mapping
each entry of a tensor \({\bs\alpha}\in\Rbb^{\bs n}\) by
\textsc{toScalarIndex}
yields the serialization
\[
\text{\textsc{Tensor}}({\bs n})\text{.\textsc{serialize}}
({\bs\alpha})\in\Rbb^{\prod_{i=1}^m n_i}.
\]
To efficiently solve the linear system \eqref{eq:LGS}, we 
require all possible matricizations \({\bs M}\in
\Rbb^{n_i\times\prod_{o\neq i}n_o}\) of \({\bs\alpha}_{\bs j}\).
The elementwise matricization is again based on the Euclidean 
division algorithm and a possible implementation is found 
in \textsc{matricize}. With a slight abuse of notation,
we refer to the entire matricization as
\[
{\bs M}=\text{\textsc{Tensor}}({\bs n})\text{.\textsc{matricize}}
({\bs\alpha},i)\in\Rbb^{n_i\times\prod_{o\neq i}n_o}.
\]
The complete computation of the sparse grid kernel 
interpolant is presented in Algorithm~\ref{algo:SGinterpolant}.

\subsection{Fast solution of the linear system of equations}\label{sec:Samplets}
It remains to provide an efficient solver for each of the 
kernel matrices \({\bs K}_{j_i}^{(i)}\), \(i=1,\ldots,m\) 
and each multi-index \({\bs j}\in \Jcal_{J}^{\bs w}\),
occurring in line 7 of Algorithm~\ref{algo:SGinterpolant}.
For this, we compute a sparse approximation to \({\bs K}_{j_i}^{(i)}\) 
by employing the samplet-based kernel matrix compression, see 
\cite{HM1,HMSS}, in combination with the sparse direct solver 
\texttt{CHOLMOD}, see \cite{CHOLMOD}. Of course, 
other approaches would be also possible here such as low-rank 
methods \cite{RR07}, adaptive low-rank methods like the multipole 
method or $\mathcal{H}$-matrices \cite{Multipol,H-Matrices}, 
fast Fourier techniques \cite{PS03}, and kernel slicing \cite{Her24}.
We decided for the samplet matrix compression as it is known to be
extremely memory efficient and a direct solver is available.
We give a brief summary of is method and refer the reader 
to \cite{HM1} for details. 

Samplets are a multiresolution basis of localized discrete signed
measures with vanishing moments, which have a natural embedding into 
RKHS by means of the Riesz isometry. Let \({\bs K}_{j_i}^{\Sigma,(i)}\)
denote the kernel matrix \({\bs K}_{j_i}^{(i)}\) in samplet coordinates
and let $\eta >0$ be a fixed parameter. Then, there holds,
\begin{equation}\label{eq:compression}
 {\Big\|{\bs K}_{j_i}^{\Sigma,(i)}-\widetilde{\bs K}_{j_i}^{\Sigma,(i)}\Big\|_F}
   \leq C(c\eta)^{-2(q+1)}{\Big\|{\bs K}_{j_i}^{\Sigma,(i)}\Big\|_F},
\end{equation}
where \(q+1\) is the number of vanishing moments, and 
\(C,c>0\) are constants. The matrix 
\(\widetilde{\bs K}_{j_i}^{\Sigma,(i)}\) is obtained from 
\({\bs K}_{j_i}^{\Sigma,(i)}\) by setting all entries to zero,
whose associated samplets have supports \(\tau,\tau'\) that satisfy 
\[
 \operatorname{dist}(\tau,\tau')\ge\eta\max\{\operatorname{diam}(\tau),
 \operatorname{diam}(\tau')\},\quad\eta>0.
\]
The compressed matrix has $\mathcal{O}(N\log N)$ remaining entries. The 
error estimate \eqref{eq:compression} is valid for \emph{asymptotically 
smooth kernels}, especially for the Mat\'ern class of kernels. For such 
kernels, the samplet compressed kernel matrices can be computed efficiently 
with loglinear cost-complexity by means of a multipole method, 
see \cite{Multipol}. In contrast to this early work, we follow 
\cite{Giebermann} and use $\mathcal{H}^2$-matrices and interpolation 
of the kernel under consideration. We refer the reader to \cite{HMQ25} 
for the description of the implementation of the particular multipole 
method we use. To further reduce the number of entries, an a-posteriori 
thresholding of small entries in \(\widetilde{\bs K}_{j_i}^{\Sigma,(i)}\) 
may be performed once the samplet compressed matrix has been assembled.
Figure~\ref{fig:KernelArithmetics} illustrates the samplet compressed 
matrix $\widetilde{\bs K}_{j_i}^{\Sigma,(i)}$, its nested dissection 
reordering (see \cite{Geo73} for details), and the resulting Cholesky 
factor in case of the exponential kernel on the unit square for 
$300\,000$ uniform random data sites.

\begin{figure}[htb]
\begin{center}
\begin{tikzpicture}
\draw(0,0)node{\includegraphics[scale=0.35,clip,trim= 0 0 0 12.25,frame]{
  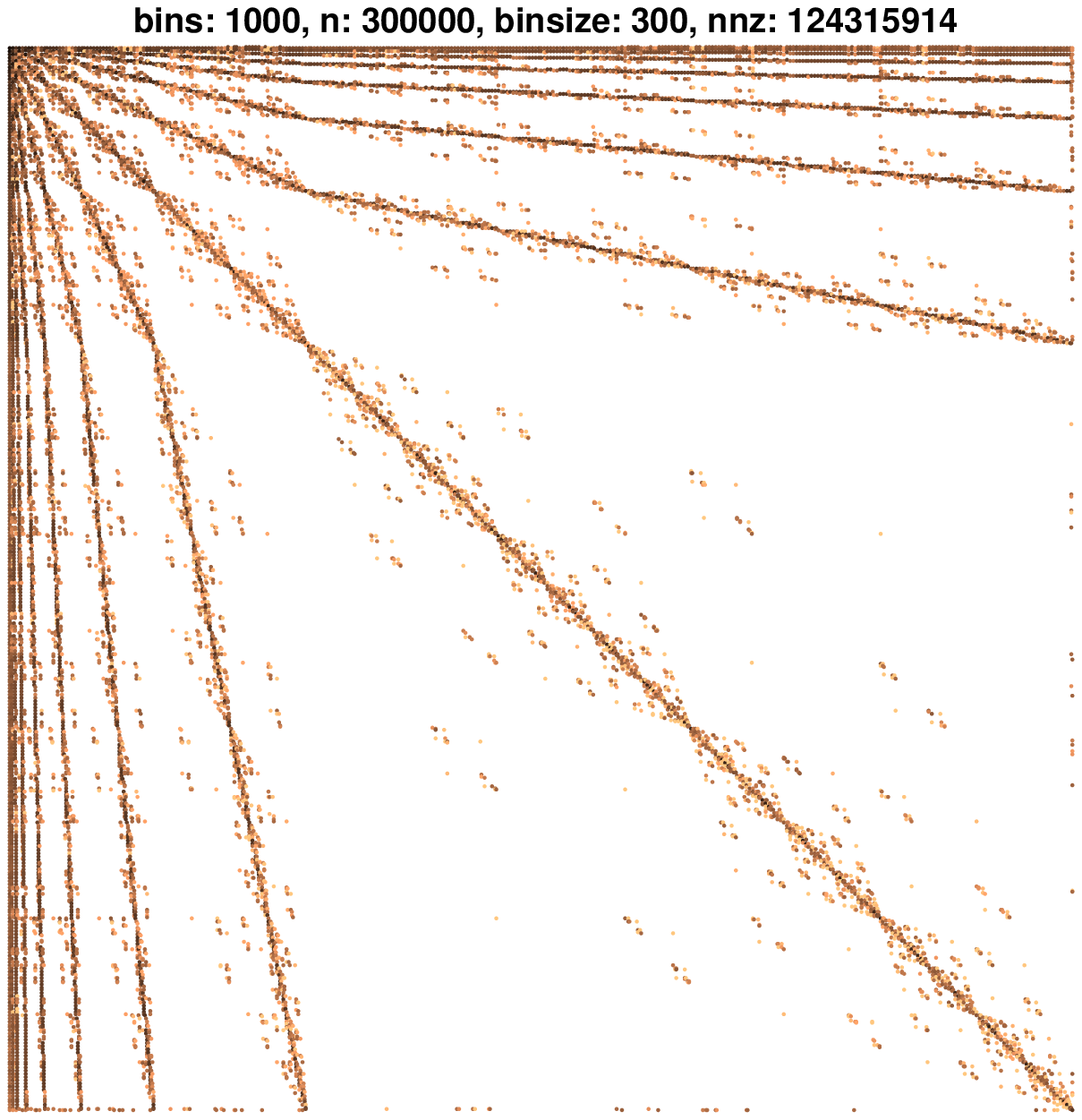}};
\draw(5,0)node{\includegraphics[scale=0.35,clip,trim= 0 0 0 12.25,frame]{
  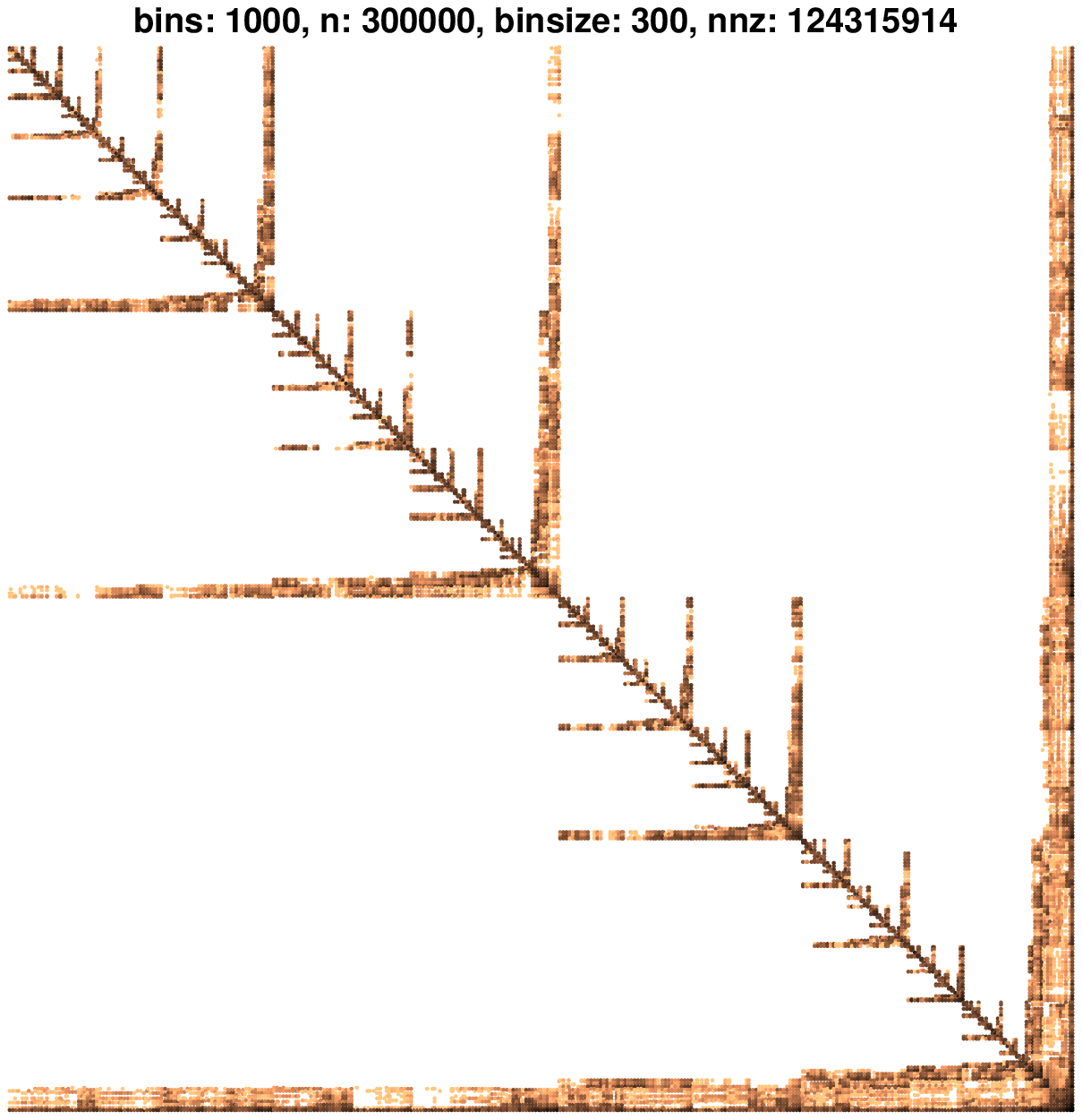}};
\draw(10,0)node{\includegraphics[scale=0.35,clip,trim= 0 0 0 12.25,frame]{
  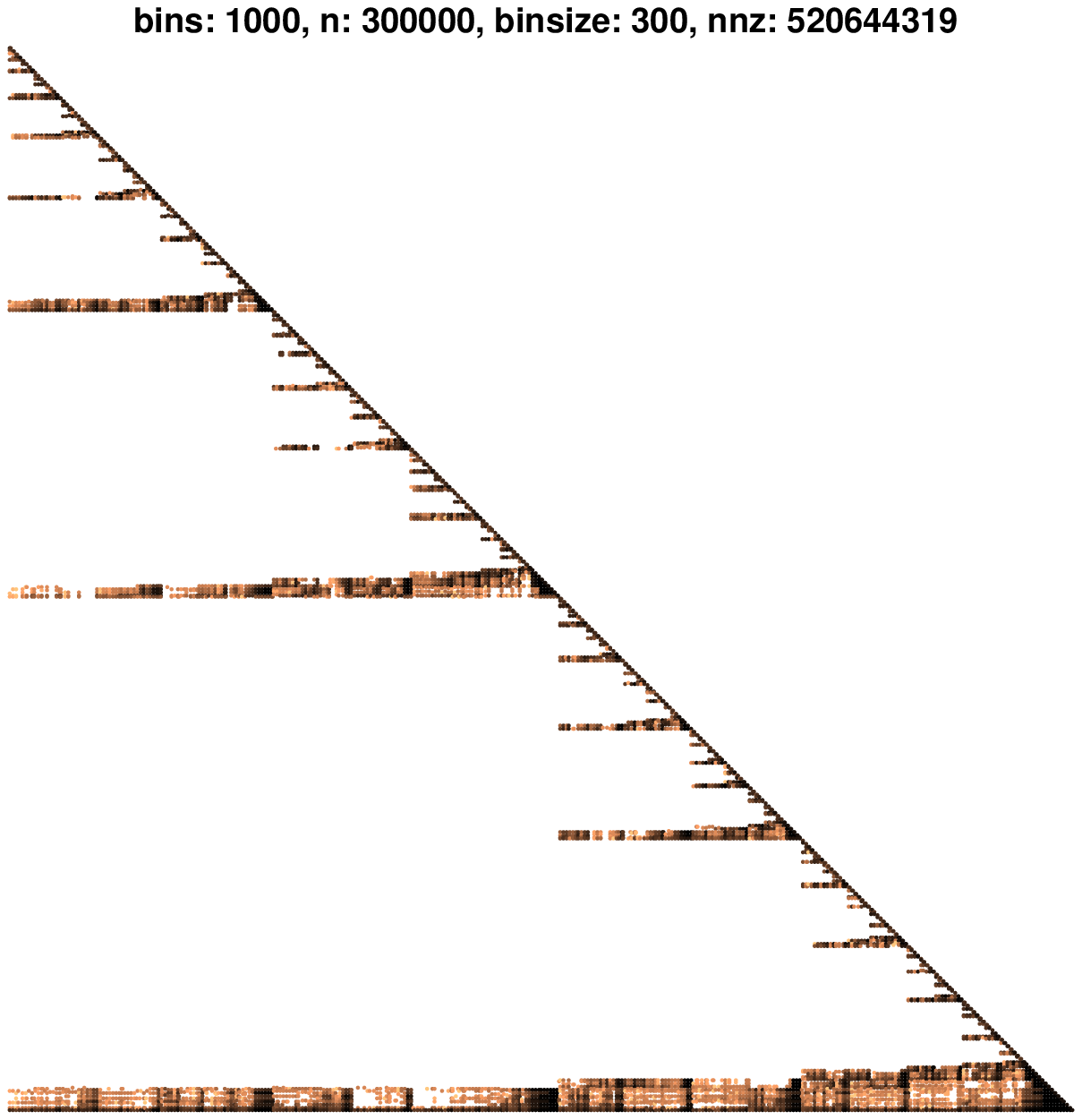}};
\end{tikzpicture}
\caption{\label{fig:KernelArithmetics}Sparsity patterns of the
samplet compressed exponential kernel on the unit square (left)
for $300\,000$ data sites, the nested dissection reordering (middle),
and the Cholesky factor (right). Each dot represents a matrix block 
of size \(300\times 300\). The number of entries per block is color 
coded, where lighter blocks have less entries.
}
\end{center}
\end{figure}

The sparse direct solver mitigates to some extent the 
computational cost for the numerical solution caused by 
the ill-conditioning of the kernel matrices for increasing 
numbers of points. Note here that the increasing condition 
number requires a corresponding increase in compression error 
accuracy to maintain a fixed overall consistency error. We 
refer to \cite{AKMW} for a detailed discussion on this matter. 

\begin{algorithm}[htb]
\caption{\label{algo:SGinterpolant}Compute Sparse Grid Kernel Interpolant}
\begin{algorithmic}[1]
\Function{Compute}{$[{\bs K}_{\bs j}]_{\bs j\in\Jcal_{J}^{\bs w}}$,
$[{\bs f}_{\bs j}]_{\bs j\in\Jcal_{J}^{\bs w}}$}
    \ForAll{${\bs j}\in\Jcal_{J}^{\bs w}$}
        \State ${\bs n}_{\bs j} \gets \big[\big|X^{(1)}_{j_1}\big|,\ldots,\big|X^{(m)}_{j_m}\big|\big]^T$
       \State ${\bs\alpha}_{\bs j}\gets\Call{Tensor(${\bs n}_{\bs j}$).serialize}{{\bs f}_{\bs j}}$

        \For{$i = 1,\ldots,m$}
        \State ${\bs M}\gets\Call{Tensor(${\bs n}_{\bs j}$).matricize}{{\bs\alpha}_{\bs j},i}$
               \State ${\bs M}\gets \big({\bs K}_{j_i}^{(i)}\big)^{-1}{\bs M}$
 \State ${\bs\alpha}_{\bs j}\gets\Call{Tensor(${\bs n}_{\bs j}$).serialize}{\bs M}$
        \EndFor
        \EndFor
    \State\Return $[{\bs \alpha}_{\bs j}]_{\bs j\in\Jcal_{J}^{\bs w}}$
\EndFunction
\end{algorithmic}
\end{algorithm}

\subsection{Evaluation of the sparse grid kernel interpolant}
Given sets of evaluation points $X^{(1)}_{\text{eval}}\subset
\Omega_1,\ldots,X^{(m)}_{\text{eval}}\subset\Omega_m$, the evaluation 
of the sparse grid kernel interpolant on the tensor product grid
$\bigtimes_{i=1}^m X^{(i)}_{\text{eval}}$ is similar to the 
solution of the interpolation problems in the sparse grid 
combination technique. The linear solver just needs to be 
replaced by a matrix-vector multiplication with the kernel 
matrices \({\bs K}^{(i)}_{X^{(i)}_{\text{eval}},X^{(i)}_{j_i}}\). 
The evaluation of the sparse grid interpolant is summarized in 
Algorithm~\ref{algo:evalSGinterpolant}. The matrix-vector 
multiplication therein can either be performed directly, in 
case of a relative small number of points in \(X^{(i)}_{j_i}\) 
or \(X^{(i)}_{\text{eval}}\), or can be sped up by means of 
the fast multipole method.

\begin{algorithm}[htb]
\caption{\label{algo:evalSGinterpolant}Evaluate Sparse Grid Kernel Interpolant}
\begin{algorithmic}[1]
\Procedure{Evaluate}{$[{\bs\alpha}_{\bs j}]_{\bs j\in\Jcal_{J}^{\bs w}}$,
$X^{(1)}_{\text{eval}},\ldots,X^{(m)}_{\text{eval}}$}
        \State ${\bs u}\gets{\bs 0}$
    \ForAll{${\bs j}\in\Jcal_{J}^{\bs w}$}
        \State ${\bs n}_{\bs j} \gets \big[\big|X^{(1)}_{j_1}\big|,\ldots,\big|X^{(m)}_{j_m}\big|\big]^T$ 
        \For{$i = 1,\ldots,m$}
        \State ${\bs M}\gets\Call{Tensor(${\bs n}_{\bs j}$).matricize}{{\bs\alpha}_{\bs j},i}$
               \State ${\bs M}\gets {\bs K}^{(i)}_{X^{(i)}_{\text{eval}},X^{(i)}_{j_i}}{\bs M}$
 \State ${\bs u}_{\bs j}\gets\Call{Tensor(${\bs n}_{\bs j}$).serialize}{{\bs M}}$
        \EndFor
        \State ${\bs u}\gets {\bs u} + c_{\bs j}^{\bs w}\,\Call{Tensor(${\bs n}_{\bs j}$).serialize}{{\bs u}_{\bs j}}$
        \EndFor
    \State\Return ${\bs u}$
\EndProcedure
\end{algorithmic}
\end{algorithm}

\section{Numerical results}
\label{sct:numerix}
\subsection{General setup}
In our numerical experiments, we employ the \emph{Mat\'ern kernels} 
or \emph{Sobolev splines} $\kappa_{\nu}\colon \mathbb{R}^d\times\mathbb{R}^d 
\to \mathbb{R}$, which are dependent on the \emph{smoothness parameter} 
$\nu > d/2$. They are defined by 
\begin{align}
 \kappa_{\nu}({\bf x},{\bf y})\isdef \frac{2^{1-\nu}}{\Gamma(\nu)} 
 r^{\nu - \frac{d}{2}} K_{\nu - \frac{d}{2}} (r),
 \quad  r \isdef\frac{1}{\sigma}\|{\bf x}-{\bs y}\|_2,
\end{align}
where $ \Gamma $ is the Riemannian gamma function and $ K_{\beta} $ 
is the modified Bessel function of the second kind, see \cite{MAT} 
for example. These kernels are known to be nonlocal and are hence not 
straightforward to deal with numerically since standard discretizations 
result in densely populated system matrices. Nonetheless, they are the 
reproducing kernels of the Sobolev spaces $H^{\nu+d/2}(\mathbb{R}^d)$, 
equipped with the canonical inner product that satisfies \eqref{eq:HsCSU}, 
and hence of great importance in practice. Although the Mat\'ern kernels 
cannot be expected to be also reproducing kernels of $H^{\nu+d/2}
(\Omega)$ with an inner product that satisfies \eqref{eq:HsCSU}, 
it turned out that they work in our numerical experiments provided that 
the interpolant is not evaluated too close to the boundary of \(\Omega\).

Throughout our experiments, we always interpolate the data 
generating process \(f\equiv 1\). At first glance, this may
appear like a very simple problem. However, for kernel interpolation
it is nontrivial, since the ansatz spaces \(\Hcal_X\)
under consideration do
not include polynomials. On the other hand, the function \(f\equiv 1\) is 
arbitrarily smooth and does not depend on the dimensionality,
which makes it a perfect test case. As mentioned in the previous 
section, the compression of smoother kernels poses a particular 
challenge in terms of accuracy. In particular for \(d=1\), we employ 
samplets with \(q+1=9\) vanishing moments and set the parameter for the cut-off 
criterion to \(\eta=5\), compare\ \cite{HM1}. In addition, an a-posteriori 
compression 
with threshold \(10^{-15}\) relative to the Frobenius norm of the compressed 
kernel matrix is performed. For \(d=2,3\), samplets with \(q+1=4\) vanishing 
moments and the parameter of the cut-off criterion set to \(\eta=2\) have been 
sufficient to maintain the overall consistency error. The threshold 
in the a-posteriori compression has been chosen as \(10^{-6}\), 
compare Section~\ref{sec:Samplets}. The length scale parameter 
of the kernels is set to \(\sigma=2\sqrt{d}\) in our examples.

All computations have been carried out on a compute server
with two AMD EPYC 7763 CPUs (64 cores each) with 2TB of main
memory and using up to 16 \texttt{OpenMP} threads if not stated
otherwise. The implementation
of the samplet matrix compression as well as of the sparse
grid combination technique are open source and available online
at \url{https://github.com/muchip/fmca}.

\subsection{Tensor product of the unit interval}
We first consider the situation 
\[
\Omega_1 = \Omega_2 = \cdots = \Omega_m = [0,1],
\]
i.e., the unit hypercube 
\(
\bigtimes_{i=1}^m \Omega_i = [0,1]^m
\)
and
\(
d_1 = d_2 = \cdots = d_m = 1.
\)
To this end, we use the tensor product kernel
\[
{\bs\kappa}\isdef\bigotimes_{i=1}^m\kappa,
\]
where \(\kappa\) is the Mat\'ern-$17/16$ kernel. The 
corresponding univariate RKHS is isomorphic to the 
Sobolev space \(H^{25/16}(0,1)\), that is, we have 
\(
s_1 = s_2 = \cdots = s_m = \frac{25}{16}.
\)
Hence, the expected univariate convergence rate with respect 
to $L^2(0,1)$ is \(25/8=3.125\) provided that the given data are 
smooth. Especially, the upper and lower bound in \eqref{eq:inequality} 
coincide and we have to choose 
\[
w_1 = w_2 = \cdots = w_m = 1. 
\]
We use the equidistant grid points 
\[
X_j = \big\{2^{-(j+1)}k:k=1,2,\ldots,2^{j+1}-1\big\},\quad j\ge 0,
\]
in the univariate directions, such that our construction computes 
the kernel interpolant with respect to the traditional $m$-variate 
sparse grid (without points at the boundary).

We first provide a benchmark on the runtime of our implementation.
Figure~\ref{fig:1DTPtiming} shows the cumulative times for the 
setup of the direct solver in samplet coordinates, computation 
of the combination technique index set, the computation of the 
coefficients and the evaluation of the interpolant at the single 
point \([1/3,\ldots,1/3]^T\in\Rbb^m\) for \(m=3,6,9,12,18\)
dimensions and \(J=1,2,\ldots,10\) levels. The combination 
technique index set \eqref{eq:anisoset} is computed up 
front using a single thread, as the computing time is 
negligible compared to the loops in line 2 of 
Algorithm~\ref{algo:SGinterpolant} and in line 3 of
Algorithm~\ref{algo:evalSGinterpolant}, respectively. 
We remark that both loops are trivial to parallelize. 
The reported times in this paragraph have been
computed by using 64 \texttt{OpenMP} 
threads with dynamic load balancing. 

The computation of all univariate direct solvers, which
takes approximately 1.6 seconds for \(J=10\), is dominating 
the overall computation time until roughly $N = 10^5$ sparse
grid points. For larger $N$, the cost for the computation 
of the coefficients of the sparse grid kernel interpolant 
and its evaluation become dominant. As can be seen in
Figure~\ref{fig:1DTPtiming}, the computation times almost 
match the theoretical loglinear rate that is caused by 
the loglinear growth number of nonzero coefficients of 
the samplet compressed kernel matrices and the matrix factors
used for the forward- and backward substitution, respectively.

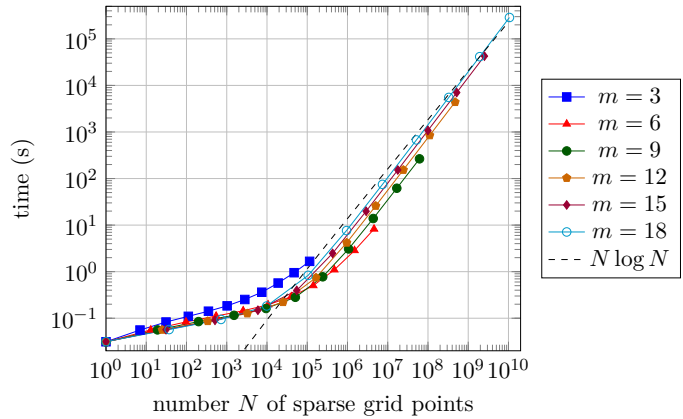
\begin{figure}[hbt]
\begin{center}
\begin{tikzpicture}[scale=0.8]
\begin{loglogaxis}[
xlabel={number $N$ of sparse grid points},
ylabel={time (s)},
grid=major,
xtick={1,1e1,1e2,1e3,1e4,1e5,1e6,1e7,1e8,1e9,1e10},
ytick={1e-1,1e0,1e1,1e2,1e3,1e4,1e5,1e6,1e7,1e8,1e9,1e10},
xmin=1,
xmax=2e10,
ymin=2e-2,
ymax=5e5,
legend style={
at={(1.05,0.5)}, 
anchor=west 
}
]
\addplot[blue, mark=square*, mark options={scale=1}] 
table[x index=1, y expr=\thisrowno{2} + \thisrowno{13}]{./Data/sg1Dtim/data_tabN.txt};
\addlegendentry{$m=3$}
\addplot[red, mark=triangle*, mark options={scale=1}] table[x index=3, y expr=\thisrowno{4} + \thisrowno{13}] {./Data/sg1Dtim/data_tabN.txt};
        \addlegendentry{$m=6$}
\addplot[green!70!black, mark=*, mark options={scale=1}] table[x index=5, y expr=\thisrowno{6} + \thisrowno{13}] {./Data/sg1Dtim/data_tabN.txt};
        \addlegendentry{$m=9$}
        \addplot[orange!80!black, mark=pentagon*, mark options={scale=1}] table[x index=7,y expr=\thisrowno{8} + \thisrowno{13}] {./Data/sg1Dtim/data_tabN.txt};
        \addlegendentry{$m=12$}
        \addplot[purple!80!black, mark=diamond*, mark options={scale=1}] table[x index=9, y expr=\thisrowno{10} + \thisrowno{13}] {./Data/sg1Dtim/data_tabN.txt};
        \addlegendentry{$m=15$}
        \addplot[cyan!80!black, mark=o, mark options={scale=1}] table[x index=11, y expr=\thisrowno{12} + \thisrowno{13}] {./Data/sg1Dtim/data_tabN.txt};
        \addlegendentry{$m=18$}
\addplot[black, dashed] table[x index=11, y expr=0.000001*\thisrowno{11}*ln(\thisrowno{11}),%
skip coords between index={0}{1}] {./Data/sg1Dtim/data_tabN.txt};
\addlegendentry{$N\log N$}
\end{loglogaxis}
\end{tikzpicture}
\caption{\label{fig:1DTPtiming}Computation times for the canonical 
sparse grid on the unit hypercube $(0,1)^m$ and \(m=3,6,9,12,15,18\).}
\end{center}
\end{figure}

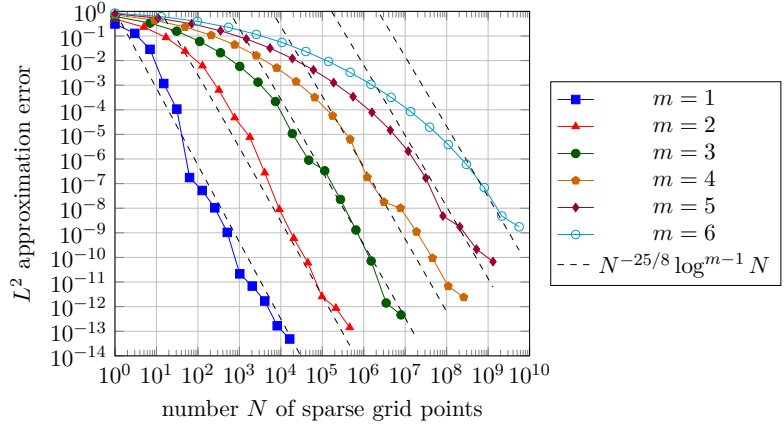
\begin{figure}[hbt]
\begin{center}
\begin{tikzpicture}[scale=0.8]
\begin{loglogaxis}[
xlabel={number $N$ of sparse grid points},
ylabel={$L^2$ approximation error},
grid=major,
xtick={1,10,10^2,10^3,10^4,10^5,10^6,10^7,10^8,10^9,10^10,10^11},
ytick={1e0,1e-1,1e-2,1e-3,1e-4,1e-5,1e-6,1e-7,
1e-8,1e-9,1e-10,1e-11,1e-12,1e-13,1e-14},
xmin=1,
xmax=1e10,
ymin=1e-14,
ymax=1,
legend style={
at={(1.05,0.5)}, 
anchor=west 
}
]
\addplot[blue, mark=square*, mark options={scale=1}] table[x index=1, y index=3,%
skip coords between index={14}{20}]{./Data/sg1D/sg1DTP_dim1.txt};
\addlegendentry{$m=1$}
\addplot[red, mark=triangle*, mark options={scale=1}] table[x index=1, y index=3,%
skip coords between index={15}{20}] {./Data/sg1D/sg1DTP_dim2.txt};
        \addlegendentry{$m=2$}
        \addplot[green!70!black, mark=*, mark options={scale=1}] table[x index=1, y index=3,%
skip coords between index={16}{20}] {./Data/sg1D/sg1DTP_dim3.txt};
        \addlegendentry{$m=3$}
        \addplot[orange!80!black, mark=pentagon*, mark options={scale=1}] table[x index=1, y index=3,%
skip coords between index={18}{20}] {./Data/sg1D/sg1DTP_dim4.txt};
        \addlegendentry{$m=4$}

        \addplot[purple!80!black, mark=diamond*, mark options={scale=1}] table[x index=1, y index=3,%
skip coords between index={18}{20}] {./Data/sg1D/sg1DTP_dim5.txt};
        \addlegendentry{$m=5$}
        \addplot[cyan!80!black, mark=o, mark options={scale=1}] table[x index=1, y index=3,%
skip coords between index={18}{20}] {./Data/sg1D/sg1DTP_dim6.txt};
        \addlegendentry{$m=6$}
\addplot[black, dashed] table[x index=1, y expr=\thisrowno{1}^(-3.125),%
skip coords between index={15}{20}] {./Data/sg1D/sg1DTP_dim1.txt};
       \addlegendentry{$N^{-25/8}\log^{m-1} N$}
\addplot[black, dashed] table[x index=1, y expr=10^3*\thisrowno{1}^(-3.125)*ln(\thisrowno{1}),%
skip coords between index={15}{20}] {./Data/sg1D/sg1DTP_dim2.txt};
\addplot[black, dashed] table[x index=1, y expr=10^7*\thisrowno{1}^(-3.125)*(ln(\thisrowno{1})^2),%
skip coords between index={17}{20}] {./Data/sg1D/sg1DTP_dim3.txt};
\addplot[black, dashed] table[x index=1, y expr=10^9*\thisrowno{1}^(-3.125)*(ln(\thisrowno{1})^3),%
skip coords between index={17}{20}] {./Data/sg1D/sg1DTP_dim4.txt};
\addplot[black, dashed] table[x index=1, y expr=10^12*\thisrowno{1}^(-3.125)*(ln(\thisrowno{1})^4),%
skip coords between index={18}{20}] {./Data/sg1D/sg1DTP_dim5.txt};
\addplot[black, dashed] table[x index=1, y expr=10^14*\thisrowno{1}^(-3.125)*(ln(\thisrowno{1})^5),%
skip coords between index={18}{20}] {./Data/sg1D/sg1DTP_dim6.txt};
\end{loglogaxis}
\end{tikzpicture}
\end{center}
\caption{\label{fig:1DTP}Convergence of the kernel interpolant
on the canonical sparse grid in $(0,1)^m$.}
\end{figure}

In Figure~\ref{fig:1DTP}, we show the convergence of the
interpolant in the $L^2$-norm, exemplarily for \(m=1,\ldots,6\). The
$L^2$-norm of the error is approximated by using a tensorized four 
point Gauss-Legendre quadrature, which exhibits \(4^m\) quadrature
points and, hence, becomes very costly in higher dimensions.
We indeed observe the theoretical 
convergence behaviour $N^{-\beta}\log^{m-1} N$ with $\beta = 25/8$ as 
predicted by Theorem~\ref{thm:cost complexity} for $t_i'=25/8$ and $t_i=0$.
Nonetheless, we also see that the constant in front of the approximation 
rate increases as the spatial dimension $m$ increases, which is a well 
known observation for sparse grid constructions.

\subsection{Tensor product of unit hypercubes in ${\bf 1+2+3}$ dimensions}
Next, we consider kernel interpolation on the unit
hypercube $[0,1]^6$ by splitting it into the product
$\Omega_1\times\Omega_2\times\Omega_3$ with
\[
\Omega_1 = [0,1],\quad \Omega_2 = [0,1]^2,\quad\Omega_3 = [0,1]^3.
\]
The tensor product kernel, which we consider, is
\[
{\bs\kappa}\isdef \kappa_1\otimes\kappa_2\otimes\kappa_3,
\]
where \(\kappa_d:\mathbb{R}^d\times\mathbb{R}^d\to\mathbb{R}\) 
is the Mat\'ern-$\big(\frac{25}{16}-\frac{d}{2}\big)$ kernel, 
$d=1,2,3$. Thus, each corresponding $d$-variate RKHS is isomorphic 
to the Sobolev space \(H^{25/16}\big([0,1]^d\big)\). Hence, the 
highest convergence rate in a particular direction $\Omega_i$
is \(25/8=3.125\).

We use a regular grid for each of the subregions, i.e., 
the interpolation points 
\[
X_j^{(1)}\isdef\big\{2^{-(j+1)}k:k=0,1,\ldots,2^{j+1}\big\},\quad j\ge 0,
\]
on $\Omega_1$ are chosen equidistantly, while $X_j^{(2)}\isdef 
\big(X_j^{(1)}\big)^2$ and $X_j^{(3)}\isdef\big(X_j^{(1)}\big)^3$.
Therefore, we have $\big|X_j^{(i)}\big| = (2^{j+1}+1)^i$ points
per level $j$ for $i=1,2,3$. In particular, there holds 
\(h_{X,\Omega_i} = 2^{-(j+1)}\sqrt{i}\) and \(q_X=2^{-j}\) 
for \(i=1,2,3\) by construction.

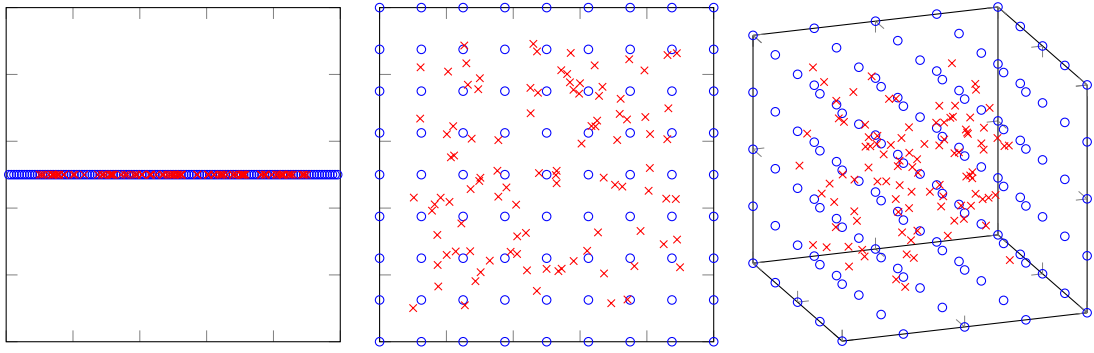
\begin{figure}[hbt]
\begin{center}
\begin{tikzpicture}
  \begin{axis}[
    width=0.4\textwidth,  
  height=0.4\textwidth,      
    xticklabels={},
    yticklabels={},
    xmin=0,xmax=1,ymin=0,ymax=1
  ]
      \addplot[ 
      only marks, 
      mark=o,
      mark options={scale=0.8},
      color=blue] table[x index=0, y expr=0.5] {./Data/sg123w/Pdim1_lvl6.txt};
      \addplot[ 
      only marks, 
      mark=x,
      mark options={scale=1},
      color=red] table[x index=0, y expr=0.5] {./Data/sg123w/Peval_dim1.txt};
  \end{axis}
\end{tikzpicture}
\begin{tikzpicture}
  \begin{axis}[
  width=0.4\textwidth,
    height=0.4\textwidth,   
      xticklabels={},
    yticklabels={},
        xmin=0,xmax=1,ymin=0,ymax=1
  ]
      \addplot[ 
      only marks, 
      mark=o,
      mark options={scale=0.8},
      color=blue] table[x index=0, y index=1] {./Data/sg123w/Pdim2_lvl4.txt};
      \addplot[ 
      only marks, 
      mark=x,
      mark options={scale=1},
      color=red] table[x index=0, y index=1] {./Data/sg123w/Peval_dim2.txt};
  \end{axis}
\end{tikzpicture}
\begin{tikzpicture}
  \begin{axis}[
  width=0.4\textwidth,  
  height=0.4\textwidth,  
  xticklabels={},
    yticklabels={},
        zticklabels={},
        xmin=0,xmax=1,ymin=0,ymax=1,zmin=0,zmax=1,
 view={70}{20}
  ]
      \addplot3[
      only marks, 
      mark=o,
      mark options={scale=0.8},
      color=blue] table[x index=0, y index=1, z index=2] {./Data/sg123w/Pdim3_lvl2.txt};
      \addplot3[ 
      only marks, 
      mark=x,
      mark options={scale=1},
      color=red] table[x index=0, y index=1, z index=2] {./Data/sg123w/Peval_dim3.txt};
  \end{axis}
\end{tikzpicture}
\end{center}
\caption{\label{fig:Setup123}Sketch of the regular grid points
(blue) and the evaluation points (red) on the unit interval, 
the unit square, and the unit cube.}
\end{figure}

\begin{figure}
\begin{center}
\begin{tikzpicture}[scale=0.8]
\begin{semilogyaxis}[
xlabel={level $j$},
ylabel={relative error},
grid=major,
xtick={0,1,2,3,4,5,6},
ytick={1e0,1e-1,1e-2,1e-3,1e-4,1e-5,1e-6,1e-7},
xmin=0,
xmax=6,
ymin=1e-7,
ymax=1,
legend style={
at={(1.05,0.5)}, 
anchor=west 
}
]
\addplot[blue, mark=triangle*, mark options={scale=1}] table[x index=0, y index=5]{./Data/sg123/benchsolve.txt};
\addlegendentry{$d=1$}
\addplot[red, mark=square*, mark options={scale=1}] table[x index=0, y index=6]{./Data/sg123/benchsolve.txt};
\addlegendentry{$d=2$}
\addplot[green, mark=pentagon*, mark options={scale=1}] table[x index=0, y index=7]{./Data/sg123/benchsolve.txt};
\addlegendentry{$d=3$}
\addplot[black, dashed] table[x index=0, y index=1]
{
0   1
6   2.2682e-06
};
\addlegendentry{$h_j^{-3.125}$}
\end{semilogyaxis}
\end{tikzpicture}
\end{center}
\caption{\label{fig:123Dbench}Convergence of the 
$d$-variate kernel approximation in case of the 
hypercube $[0,1]^d$ for $d=1,2,3$.}

\end{figure}
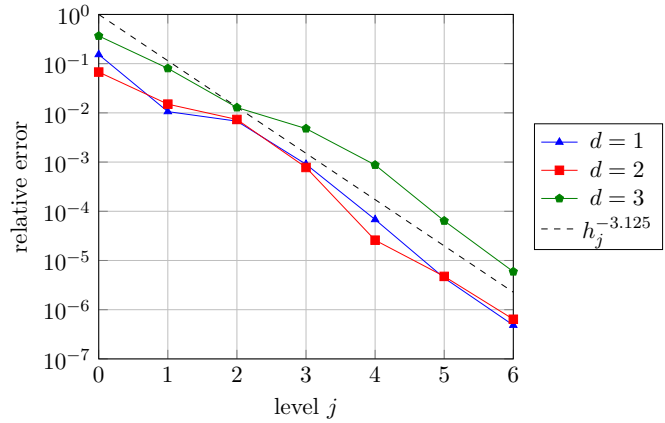
\begin{figure}
\begin{center}
\begin{tikzpicture}[scale=0.8]
\begin{loglogaxis}[
xlabel={number $N$ of degrees of freedom},
ylabel={relative error},
grid=major,
xtick={1e0,1e1,1e2,1e3,1e4,1e5,1e6,1e7,1e8},
ytick={1e0,1e-1,1e-2,1e-3,1e-4,1e-5,1e-6,1e-7},
xmin=0.1,
xmax=1e8,
ymin=1e-5,
ymax=1,
legend style={
at={(1.05,0.5)}, 
anchor=west 
}
]
\addplot[blue, mark=triangle*, mark options={scale=1}] table[x index=1, y index=2]{./Data/sg123/sparseGrid123.txt};
\addlegendentry{accuracy equilibrated}
\addplot[green, mark=square*, mark options={scale=1}] table[x index=1, y index=2,%
skip coords between index={18}{20}]{./Data/sg123w_work/sparseGrid123w.txt};
\addlegendentry{cost-benefit equilibrated}
\addplot[red, mark=pentagon*, mark options={scale=1}] table[x index=1, y index=2,%
skip coords between index={18}{20}]{./Data/sg123w/sparseGrid123w.txt};
\addlegendentry{DoF equilibrated}
\addplot[black, dashed] table[x index=0, y index=1]
{
10   100
1e8  4.6416e-07
};
\addlegendentry{$N^{-1.04}$}
\end{loglogaxis}
\end{tikzpicture}
\end{center}
\caption{\label{fig:123D}Convergence rates of the kernel approximant 
with respect to different sparse grids on $(0,1)\times (0,1)^2\times 
(0,1)^3$.}
\end{figure}
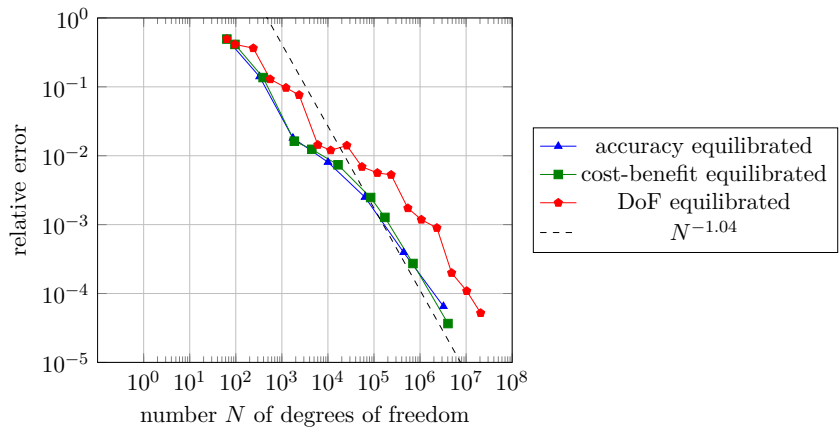

After the kernel interpolant has been computed, it is 
evaluated at 100 uniform random points for each subregion
$\Omega_i$, located in a hypercube of distance 0.1 from the 
respective subregion's boundary. 
We refer to 
Figure~\ref{fig:Setup123} for a visualization of the presented 
setup. Therein, the evaluation points are indicated in red.

Next, we consider the $d$-variate approximation for $d = 1,2,3$ 
to validate the appropriate choice of the number of vanishing 
moments of the samplets and the compression parameter $\eta$ 
for the matrix compression, and thus for our solver for the 
different subproblem directions. The convergence of the approximant 
with respect to each particular subregion $\Omega_i$ is shown in 
Figure~\ref{fig:123Dbench}. Indeed, we observe the convergence 
rate $h_j^{-3.125}$ in all three case as predicted, so that 
we can be sure that the compression works correctly.

We next consider the kernel interpolation of the respective
sparse grid. For the present setting, we can summarize the 
parameters as
\begin{equation}\label{eq:parameters_1}
d_1 = 1, \quad d_2 = 2, \quad d_3 = 3,
\quad s_1 = s_2 = s_3 = \frac{25}{16}.
\end{equation}
Therefore, choosing the weights 
\begin{equation}\label{eq:parameters_2}
w_1 = w_2 = w_3 = 1
\end{equation}
for the sparse grid construction equilibrates the accuracies 
in the particular directions, while choosing the weights 
\begin{equation}\label{eq:parameters_3}
w_1 = 1/3, \quad w_2 = 2/3, \quad w_3 = 1
\end{equation}
equilibrates their degrees of freedom. For the equlibration 
of the cost-benefit-rate, we have to choose 
\begin{equation}\label{eq:parameters_4}
w_1 = 33/49, \quad w_2 = 33/41, \quad w_3 = 1.
\end{equation}
The resulting convergence rates with respect to the number 
$N$ of the degrees of freedom are given in Figure~\ref{fig:123D}.
The expected rate of convergence is $N^{-\beta}$ with $\beta 
= 25/24 \approx 1.04$ up to polylogarithmic terms. Indeed, 
after some preasymptotic regime, we observe the predicted
convergence rate of $N^{-1,04}$.

\subsection{Tensor product of general regions in ${\bf 1+2+3}$ dimensions}
\label{sct:bunny}
In our final numerical experiment, we consider the tensor
product of uniformly chosen random points on the unit interval 
$\Omega_1 = [0,1]$, of uniformly chosen random points on the sphere
$\Omega_2 = \mathbb{S}^2$, and the nodal points of a tetrahedral
mesh of a rabbit $\Omega_3 \subset\mathbb{R}^3$ (involving 
three-dimensional points at the surface of the well-known 
Stanford bunny and in the interior of the bunny). We refer to 
Figure~\ref{fig:Setup123_geos} for an illustration of this 
geometrical situation. 

Table~\ref{tab:pts123equi} lists the number of points per level 
for each of the geometries and all considered combinations.
As can be seen, when proceeding from level \(j\) to \(j+1\),
the number of points approximately doubles on the interval.
For the sphere, which is a two-dimensional manifold, i.e., 
$d_2 = 2$, we asymptotically observe the factor four. Moreover, 
the number of points of the rabbit grows with a factor about 
6--8.

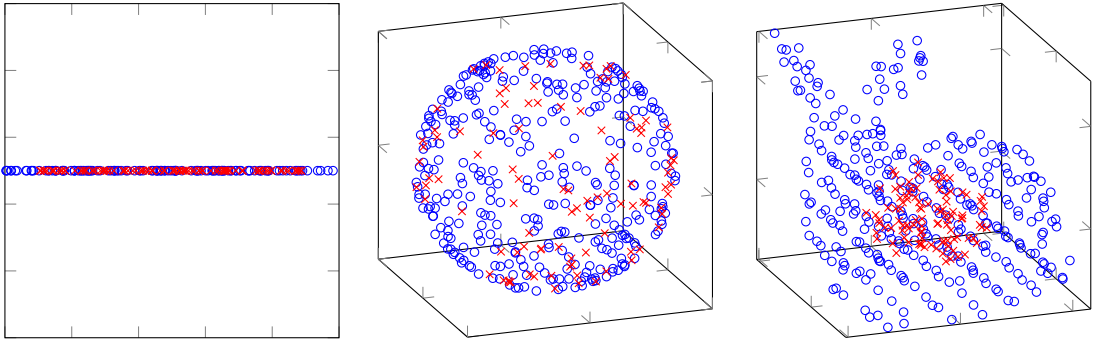
\begin{figure}
\begin{center}
\begin{tikzpicture}
  \begin{axis}[
    width=0.4\textwidth,  
    height=0.4\textwidth,      
    xticklabels={},
    yticklabels={},
    xmin=0,xmax=1,ymin=0,ymax=1
  ]
      \addplot[ 
      only marks, 
      mark=o,
      mark options={scale=0.8},
      color=blue] table[x index=0, y expr=0.5] {./Data/sg123w_geos/Pdim1_lvl6.txt};
      \addplot[ 
      only marks, 
      mark=x,
      mark options={scale=1},
      color=red] table[x index=0, y expr=0.5] {./Data/sg123w_geos/Peval_dim1.txt};
  \end{axis}
\end{tikzpicture}
\begin{tikzpicture}
  \begin{axis}[
  width=0.4\textwidth,
    height=0.4\textwidth,   
      xticklabels={},
    yticklabels={},
        zticklabels={},
        xmin=-1,xmax=1,ymin=-1,ymax=1,zmin=-1,zmax=1,
                 view={70}{20}
  ]
      \addplot3[ 
      only marks, 
      mark=o,
      mark options={scale=0.8},
      color=blue] table[x index=0, y index=1, z index=2] {./Data/sg123w_geos/Pdim2_lvl3.txt};
      \addplot3[ 
      only marks, 
      mark=x,
      mark options={scale=1},
      color=red] table[x index=0, y index=1, z index=2] {./Data/sg123w_geos/Peval_dim2.txt};
  \end{axis}
\end{tikzpicture}
\begin{tikzpicture}
  \begin{axis}[
  width=0.4\textwidth,  
  height=0.4\textwidth,  
  xticklabels={},
    yticklabels={},
        zticklabels={},
        xmin=-0.52,xmax=0.33,ymin=-0.51,ymax=0.58,zmin=-0.39,zmax=0.72,
          view={70}{20}
  ]
      \addplot3[
      only marks, 
      mark=o,
      mark options={scale=0.8},
      color=blue] table[x index=2, y index=0, z index=1] {./Data/sg123w_geos/Pdim3_lvl3.txt};
      \addplot3[ 
      only marks, 
      mark=x,
      mark options={scale=1},
      color=red] table[x index=0, y index=1, z index=2] {./Data/sg123w_geos/Peval_dim3.txt};
  \end{axis}
\end{tikzpicture}
\caption{\label{fig:Setup123_geos}Sketch of the quasi-uniform 
points (blue) and evaluation points (red) on the unit interval, 
the unit sphere, and the Stanford bunny.}
\end{center}
\end{figure}

On the particular subregions $\Omega_i$, we have
unstructured, quasi-uniform data sites, which we coarsen 
by employing Algorithm~\ref{alg:UnifSubs} as given in 
Subsection~\ref{sct:coarsening}. 
On the unit interval, we start from a
point set with $4\,319\,030$ points,
a separation distance of $5.32\cdot 10^{-14}$
and a fill distance of $2.62\cdot 10^{-6}$, 
while on the sphere, we start from a
point set with $2\,879\,320$ points,
a separation distance of $8.13\cdot 10^{-7}$
and a fill distance of $5.09\cdot 10^{-3}$,
and finally on
the rabbit, we start from a
point set with $1\,439\,610$ points,
a separation distance of $5.25\cdot 10^{-4}$
and a fill distance of $8.93\cdot 10^{-3}$.
It can be seen from Table~\ref{tab:qhscattered} that 
the fill distance \(h_{X_j,X}\), which we consider an 
approximation of \(h_{X_j,\Omega}\), approximately halves 
with respect to the level in each particular example, as 
desired. On the other hand, the separation distance stays 
proportional to the fill distance. 
For the sphere we remark that the separation distance and
the fill distance have been approximated using the Euclidean norm.
Therefore, we have
a nested sequence of sets $X_j^{(i)}$ of data sites 
which satisfy $\big|X_j^{(i)}\big|\sim 2^{j i}$, 
$i=1,2,3$.

Moreover, after the sparse grid kernel interpolant is 
computed, it is evaluated at the product of randomly 
distributed points $X^{(i)}_{\text{eval}}\subset\Omega_i$. 
These are, in case of the interval and the rabbit, again 
chosen with a certain distance from the boundary.
We refer again to Figure~\ref{fig:Setup123_geos} 
for a visualization. The convergence of the univariate solvers 
is shown in Figure~\ref{fig:123D_goesbench}. As can be seen, 
all of them achieve the expected convergence rate of $h_j^{-3.125}$.

\begin{table}[htb]
\begin{center}
\begin{tabular}{|c|c|c|c| }
        \cline{2-4}
         \multicolumn{1}{c|}{}
         & Interval & Sphere & Rabbit \\ \hline
         $j=0$ & 1 & 1 & 1 \\ \hline
         $j=1$ & 3 & 9 & 9 \\ \hline
         $j=2$ & 7 & 65 & 58 \\ \hline
         $j=3$ & 15 & 337 & 326 \\ \hline
         $j=4$ & 31 & 1497 & 1933 \\ \hline
         $j=5$ & 63 & 6246 & 12482 \\ \hline
         $j=6$ & 127 & 24952 & 88489 \\ \hline
         $j=7$ & 255 & 97224 & --- \\ \hline
         $j=8$ & 511 & --- & --- \\ \hline
         $j=9$ & 1023 & --- & --- \\ \hline
         $j=10$ & 2047 & --- & --- \\ \hline
         $j=11$ & 4095 & --- & --- \\ \hline
         $j=12$ & 8191 & --- & --- \\ \hline
         $j=13$ & 16383 &--- & --- \\ \hline
         $j=14$ & 32767 & --- & --- \\ \hline
         $j=15$ & 65535 & --- & --- \\ \hline
         $j=16$ & 131071 & --- & --- \\ \hline
         $j=17$ & 262143 & --- & --- \\ \hline
\end{tabular}
\end{center}
\caption{\label{tab:pts123equi}Numbers $N$ of points 
per level that enter the sparse grid construction
for the interval ($d=1$), the sphere ($d=2$), and 
the rabbit ($d=3$).}
\end{table}

\begin{table}
\begin{center}
\begin{tabular}{|c|c|c|c|c|c|c|}
\cline{2-7}
 \multicolumn{1}{c|}{}
 & \multicolumn{2}{c|}{Interval} & 
\multicolumn{2}{c|}{Sphere} &
\multicolumn{2}{c|}{Rabbit} \\\cline{2-7}
 \multicolumn{1}{c|}{}
 & $q_{X_j}$ & $h_{X_j,X_J}$ & $q_{X_j}$ & $h_{X_j,X_J}$ & $q_{X_j}$ & $h_{X_j,X_J}$ \\\hline
$j=0$ & ---                 & $4.92\cdot 10^{-1}$ & ---                 & $2.00$              & ---                 & $8.34\cdot{10^{-1}}$ \\
$j=1$ & $2.50\cdot 10^{-1}$ & $2.50\cdot 10^{-1}$ & $5.55\cdot 10^{-1}$ & $1.16$              & $3.57\cdot 10^{-1}$ & $4.83\cdot 10^{-1}$ \\
$j=2$ & $1.25\cdot 10^{-1}$ & $1.25\cdot 10^{-1}$ & $7.94\cdot 10^{-4}$ & $4.51\cdot 10^{-1}$ & $2.49\cdot 10^{-2}$ & $2.41\cdot 10^{-1}$ \\
$j=3$ & $6.25\cdot 10^{-2}$ & $6.25\cdot 10^{-2}$ & $6.23\cdot 10^{-4}$ & $2.42\cdot 10^{-1}$ & $5.11\cdot 10^{-3}$ & $1.22\cdot 10^{-1}$ \\
$j=4$ & $3.12\cdot 10^{-2}$ & $3.12\cdot 10^{-2}$ & $4.22\cdot 10^{-5}$ & $1.09\cdot 10^{-1}$ & $4.04\cdot 10^{-3}$ & $6.43\cdot 10^{-2}$ \\
$j=5$ & $1.56\cdot 10^{-2}$ & $1.56\cdot 10^{-2}$ & $4.22\cdot 10^{-5}$ & $5.52\cdot 10^{-2}$ & $2.87\cdot 10^{-3}$ & $3.41\cdot 10^{-2}$ \\
$j=6$ & $7.81\cdot 10^{-3}$ & $7.81\cdot 10^{-3}$ & $4.22\cdot 10^{-5}$ & $2.92\cdot 10^{-2}$ & $1.99\cdot 10^{-3}$ & $1.83\cdot 10^{-2}$ \\\hline
\end{tabular}
\end{center}
\caption{\label{tab:qhscattered}Separation distance and fill 
distance for the different geometries.}
\end{table}

The parameters for the construction of the sparse grid 
are the same as in the previous experiment, i.e., the 
parameters are given as in \eqref{eq:parameters_1} for 
the underlying approximation spaces. Therefore, the 
accuracy-equilibrated sparse grid is given by the weights in 
\eqref{eq:parameters_2}, the degrees-of-freedom-equilibrated 
sparse grid is given by the weights in \eqref{eq:parameters_3}, 
and finally the cost-benefit-equilibrated sparse grid is given 
by the weights in \eqref{eq:parameters_4}. As can be inferred 
from Figure~\ref{fig:123D_goes}, the different settings produce 
essentially the same convergence rate, which indeed shows the 
$N^{-1.04}$ behavior as the number $N$ of sparse grid points
increases.

\begin{figure}
\begin{center}
\begin{tikzpicture}[scale=0.8]
\begin{semilogyaxis}[
xlabel={level $j$},
ylabel={relative error},
grid=major,
xtick={0,1,2,3,4,5,6},
ytick={1e0,1e-1,1e-2,1e-3,1e-4,1e-5,1e-6,1e-7,1e-8},
xmin=0,
xmax=6,
ymin=1e-8,
ymax=1,
legend style={
at={(1.05,0.5)}, 
anchor=west 
}
]
\addplot[blue, mark=triangle*, mark options={scale=1}] table[x index=0, y index=5]{./Data/sg123_geos/benchsolve.txt};
\addlegendentry{$d=1$}
\addplot[red, mark=square*, mark options={scale=1}] table[x index=0, y index=6]{./Data/sg123_geos/benchsolve.txt};
\addlegendentry{$d=2$}
\addplot[green, mark=pentagon*, mark options={scale=1}] table[x index=0, y index=7]{./Data/sg123_geos/benchsolve.txt};
\addlegendentry{$d=3$}
\addplot[black, dashed] table[x index=0, y index=1]
{
0   1
6   2.2682e-06
};
\addlegendentry{$h_j^{-3.125}$}
\end{semilogyaxis}
\end{tikzpicture}
\end{center}
\caption{\label{fig:123D_goesbench}Convergence of the 
kernel interpolant on the interval ($d=1$), the sphere ($d=2$),
and the rabbit ($d=3$).}
\end{figure}
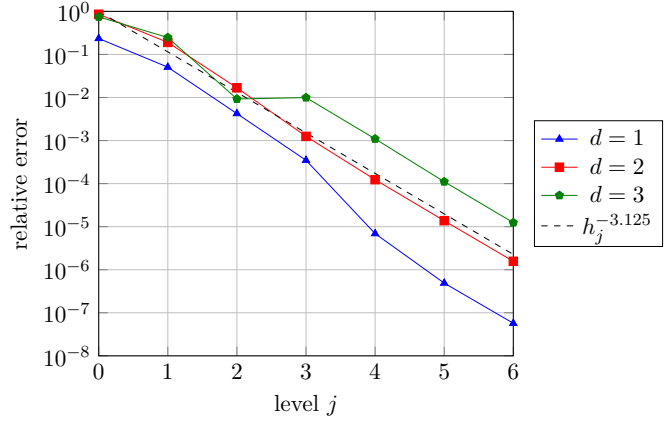
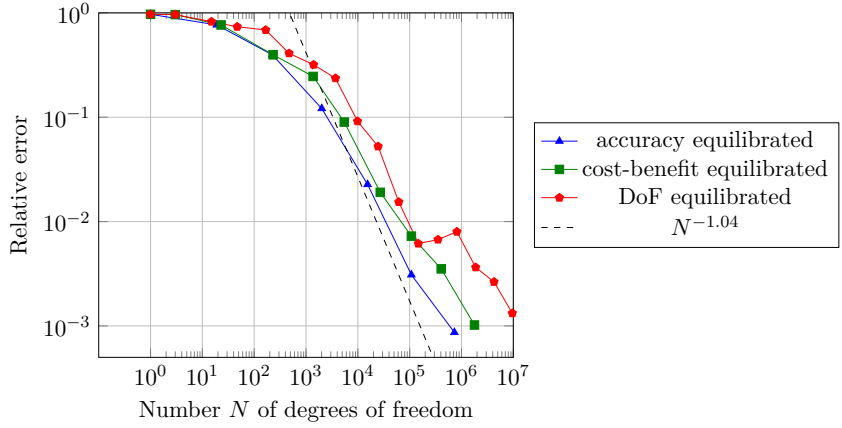
\begin{figure}
\begin{center}
\begin{tikzpicture}[scale=0.8]
\begin{loglogaxis}[
xlabel={Number $N$ of degrees of freedom},
ylabel={Relative error},
grid=major,
xtick={1e0,1e1,1e2,1e3,1e4,1e5,1e6,1e7,1e8},
ytick={1e0,1e-1,1e-2,1e-3,1e-4,1e-5,1e-6,1e-7},
xmin=0.1,
xmax=1e7,
ymin=5e-4,
ymax=1,
legend style={
at={(1.05,0.5)}, 
anchor=west 
}
]
\addplot[blue, mark=triangle*, mark options={scale=1}] table[x index=1, y index=2]{./Data/sg123_geos/sparseGrid123.txt};
\addlegendentry{accuracy equilibrated}
\addplot[green, mark=square*, mark options={scale=1}] table[x index=1, y index=2,%
skip coords between index={18}{20}]{./Data/sg123w_geos_work/sparseGrid123w.txt};
\addlegendentry{cost-benefit equilibrated}
\addplot[red, mark=pentagon*, mark options={scale=1}] table[x index=1, y index=2,%
skip coords between index={18}{20}]{./Data/sg123w_geos/sparseGrid123w.txt};
\addlegendentry{DoF equilibrated}
\addplot[black, dashed] table[x index=0, y index=1]
{
10   100
1e8  4.6416e-07
};
\addlegendentry{$N^{-1.04}$}
\end{loglogaxis}
\end{tikzpicture}
\end{center}
\caption{\label{fig:123D_goes}Convergence rates of the kernel approximant 
with respect to different sparse grids on the product of general subregions
in $1+2+3$ dimensions.}
\end{figure}

\section{Conclusion}\label{sct:conclusio}
In the present article, we have considered kernel interpolation 
on sparse grids in Sobolev spaces of dominating mixed derivatives. 
We have discussed the optimal construction of the sparse grid in case 
of product regions of arbitrary dimension and of arbitrary smoothness 
with respect to the particular regions. Especially, we derived
improved estimates on the approximation error, using duality 
arguments, provided that the function to be interpolated exhibits 
additional smoothness. Our convergence analysis is based 
entirely on the doubling trick. If the doubling trick does not 
apply, we are only allowed to choose ${\bs t}' = {\bs s}$ in 
Section~\ref{sec:multivariate}. Our analysis, however, can be
adopted to this case by obvious modifications. Specifically, 
the result of Theorem \ref{thm:accuracy} becomes
\[
  \big\|f-\widehat{\bs P}_J^{\bs w}f\big\|_{{\bs H}^{\bs t}(\bs\Omega)}
    \lesssim 2^{-J\min\{\frac{s_1-t_1}{w_1},\ldots,\frac{s_m-t_m}{w_m}\}} 
    J^{P-1}\|f\|_{{\bs H}^{\bs s}(\bs\Omega)},
    \quad {\bs 0}\le {\bs t}<{\bf s}.
\]
Consequently, Theorem \eqref{thm:cost complexity} 
then reads
\[
\big\|f-\widehat{\bs P}_J^{\bs w}f\big\|_{{\bs H}^{\bs t}(\bs\Omega)}
  	\lesssim N^{-\beta}(\log N)^{(P-1)+\beta(R-1)}
		\|f\|_{{\bs H}^{\bs s}(\bs\Omega)},
    \quad {\bs 0}\le {\bs t}<{\bf s},
\]
with
\[
  \beta \isdef \frac{\min\{(s_1-t_1)/w_1,\ldots,(s_m-t_m)/w_m\}}
  		{\max\{d_1/w_1,\ldots,d_m/w_m\}}.
\]

For the numerical solution of the interpolation problem,
we have proposed an efficient algorithm that combines the
sparse grid combination technique with a fast direct solver for 
nonlocal operators on the subproblems. We presented the results 
of numerical experiments in up to 18 dimensions and with billions 
of degrees of freedoms in the sparse grid, which validate 
the presented theory. We emphasize that the problem size would 
have been restricted seriously without the application of an 
efficient method for dealing with the nonlocal kernel matrices,

We finally point out that the proposed sparse grid kernel 
interpolation is also applicable with straightforward modification 
when dimension weights are present. In this case, the logarithmic 
factors might be removed and even dimension-robustness can be achieved 
provided that the weights decay sufficiently fast. Such a situation is 
typically found in uncertainty quantification or machine learning, 
see \cite{Dung2,Zech} for example.

\appendix
\section{An inner product for the doubling trick}\label{app:A}
We construct here an inner 
product in $H^s(\Omega)$ which satisfies the 
assumption \eqref{eq:HsCSU}. The resulting
reproducing kernel then enables the doubling 
trick from Lemma~\ref{lem:doubling} that we exploit 
in the analysis of Section~\ref{sec:multivariate}.

\begin{lemma}\label{lem:existenceCSU}
Let $\Omega\subset\mathbb{R}^d$ be a Lipschitz domain.
Then, there exists an inner product \((\cdot,\cdot)_E\)
on \(H^s(\Omega)\) such that
\[
  (u,v)_{E}
  	\lesssim \|u\|_{L^2(\Omega)} \|v\|_{H^{2s}(\Omega)}
\]
for all $u\in H^s(\Omega)$ and $v\in H^{2s}(\Omega)$.
\end{lemma}
\begin{proof}
Let \(E\colon H^r(\Omega)\to H^r(\Rbb^d)\),
\(0\leq r\leq 2s\), be a uniform extension operator, i.e.,
\[
\|Eu\|_{H^r(\Rbb^d)}\leq C\|u\|_{H^r(\Omega)}
\quad\text{for all }0\leq r\leq 2s
\]
for some \(C>0\).
A suitable extension operator is the one introduced 
by Rychkov in \cite{Rychkov} for example. We set
\[
  (u,v)_{E}\isdef (Eu,Ev)_{H^r(\Rbb^d)}\quad\text{for all }0\leq r\leq 2s.
\]
Especially, we have
\[
  (u,v)_{E}\leq\|Eu\|_{H^r(\Rbb^d)}
  \|Ev\|_{H^r(\Rbb^d)}\leq C^2\|u\|_{H^r(\Omega)}\|v\|_{H^r(\Omega)}.
\]
Therefore, the bilinear form is continuous.
Similarly, we find by the monotonicity of the integral that
\[
\|u\|_{H^r(\Omega)}^2\leq\|Eu\|_{H^r(\Rbb^d)}^2
=(Eu,Eu)_{H^r(\Rbb^d)}=(u,u)_{E}
\]
due to \(Eu|_\Omega= u\), which shows the ellipticity. 
As a consequence, the bilinear form \((\cdot,\cdot)_E\) 
defines an inner product on \(H^r(\Omega)\) for \(0\leq r\leq s\) 
and an equivalent norm. Finally, there holds by Plancherel's 
theorem that
\begin{align*}
  (u,v)_{E}&= (Eu,Ev)_{H^s(\Rbb^d)}=
  \int_{\Rbb^d}\widehat{Eu}
  \overline{\widehat{Ev}}(1+\|{\bs\xi}\|_2^2)^s\operatorname{d}\!{\bs\xi}\\
  &\leq\sqrt{\int_{\Rbb^d}\big|\widehat{Eu}\big|^2\operatorname{d}\!{\bs\xi}}
    \sqrt{\int_{\Rbb^d}
  \big|\widehat{Ev}\big|^2(1+\|{\bs\xi}\|_2^2)^{2s}\operatorname{d}\!{\bs\xi}}\\
  &=\|Eu\|_{L^2(\Rbb^d)}\|Ev\|_{H^{2s}(\Rbb^d)}\leq C^2\|u\|_{L^2(\Omega)}\|v\|_{H^{2s}(\Omega)}.
\end{align*}
\end{proof}

In view of the previous lemma, the operator 
\(A\isdef E^\star E\colon H^s(\Omega)\to H^s(\Omega)\) 
is a symmetric, elliptic and continuous operator with
\[
(u,v)_{E}=(Au,v)_{H^s(\Omega)}.
\]
With respect to the \((\cdot,\cdot)_{E}\) inner product, 
we obtain for the reproducing kernel
\[
u(y)=\big(\kappa({\cdot},y),u\big)_{E}=\big({A}\kappa({\cdot},y),u\big)_{H^s(\Omega)}
\]
and, therefore,
\[
\kappa({\cdot},y)=A^{-1}R\delta_y,
\]
where \(R\colon [H^s(\Omega)]'\to H^s(\Omega)\) is the Riesz
isometry with respect to the \(H^s(\Omega)\)-inner product.

\section*{Acknowledgments}
Michael Griebel was supported by the \emph{Hausdorff Center for Mathematics} 
(HCM) in Bonn, funded by the Deutsche Forschungsgemeinschaft (DFG, German 
Research Foundation) under Germany's Excellence Strategy -- EXC-2047/1 -- 
390685813 and by the CRC 1060 \emph{The Mathematics of Emergent Effects} 
-- 211504053 of the Deutsche Forschungsgemeinschaft. Helmut Harbrecht was 
funded in parts by the Swiss National Science Foundation (SNSF) through the 
Vietnamese-Swiss Joint Research Project IZVSZ2\_229568. Michael Multerer
was funded in parts by the Swiss Federal Office of Energy SFOE 
as part of the SWEET project SURE and by the SNSF starting grant 
``Multiresolution methods for unstructured data'' (TMSGI2\_211684).
Finally, the authors thank Felix Bartel,
especially with respect to Subsection~\ref{sct:comparison},
R\"udiger Kempf and Holger Wendland
for fruitful discussions and helpful remarks.

\end{document}